\documentclass[a4paper,11pt]{amsart}
\usepackage[backend=bibtex,style=numeric,giveninits=true]{biblatex}
\renewbibmacro*{issue+date}{%
  \ifboolexpr{not test {\iffieldundef{year}} or not test {\iffieldundef{issue}}}
    {\printtext[parens]{%
       \iffieldundef{issue}
         {\usebibmacro{date}}
         {\printfield{issue}%
          \setunit*{\addspace}%
          \usebibmacro{date}}}}
    {}%
  \newunit}
\renewbibmacro{in:}{}
\usepackage{amsmath}
\usepackage{amsthm,amsfonts,amssymb,graphics}
\usepackage{pgfplots}

\usepackage{tikz}
\usetikzlibrary{patterns}

\usepackage{hyperref}
\usepackage{color}
\usepackage[normalem]{ulem}
\usepackage{enumerate}
\usepackage{a4wide}
\usepackage{multirow}
\usepackage{booktabs}
\usepackage[foot]{amsaddr}
\usetikzlibrary{calc}
\usepackage{dsfont}
\usepackage{subfig}
\usetikzlibrary{calc}

\newtheorem{theorem}{Theorem}[section]
\newtheorem{lemma}[theorem]{Lemma}
\newtheorem{proposition}[theorem]{Proposition}
\newtheorem{corollary}[theorem]{Corollary}

\newtheorem{condition}[theorem]{Conditions}
\theoremstyle{remark}
\newtheorem{remark}[theorem]{Remark}
\newtheorem{definition}[theorem]{Definition}

\numberwithin{equation}{section}

\newcommand \id{\mathds 1}
\newcommand {\R} {\mathbb{R}}

\begin{document}

\title[On the number of excursion sets of Planar Gaussian fields]{On the number of excursion sets of Planar Gaussian fields}
\author{Dmitry Beliaev\textsuperscript{1}}
\address{\textsuperscript{1}Mathematical Institute, University of Oxford}
\email{belyaev@maths.ox.ac.uk}
\author{Michael McAuley\textsuperscript{1}}
\email{mcauley@maths.ox.ac.uk}
\author{Stephen Muirhead\textsuperscript{2}}
\address{\textsuperscript{2}Department of Mathematics, King's College London\\
	\emph{Present address: School of Mathematical Sciences, Queen Mary University of London}}
\email{s.muirhead@qmul.ac.uk}
\thanks{The first author was supported by the Engineering \& Physical Sciences Research Council (EPSRC) Fellowship EP/M002896/1. The third author was supported by the EPSRC Grant EP/N009436/1 ``The many faces of random characteristic polynomials''. The authors would like to thank Manjunath Krishnapur for pointing out~\cite{swerling1962statistical} and also Igor Wigman and Mikhail Sodin for useful comments and suggestions.}


\begin{abstract}
The Nazarov-Sodin constant describes the average number of nodal set components of smooth Gaussian fields on large scales. We generalise this to a functional describing the corresponding number of level set components for arbitrary levels. Using results from Morse theory, we express this functional as an integral over the level densities of different types of critical points, and as a result deduce the absolute continuity of the functional as the level varies. We further give upper and lower bounds showing that the functional is at least bimodal for certain isotropic fields, including the important special case of the random plane wave.
\end{abstract}

\maketitle

\section{Introduction}\label{s:introduction}
\subsection{The Nazarov-Sodin constant}
Let $f: \mathbb{R}^2 \rightarrow \mathbb{R}$ be a continuous stationary planar Gaussian field normalised to have zero mean and unit variance. The \emph{nodal set} of $f$ is the random set
\begin{displaymath}
	\mathcal{N}=\left\{x\in\R^2:f(x)=0\right\} .
\end{displaymath}
Let $\kappa:\R^2\rightarrow[-1,1]$ denote the covariance kernel of $f$, i.e.\ $\kappa(x)=\mathbb{E}[f(x)f(0)]$. We assume throughout that $\kappa$ is $C^{4+}$, which ensures that almost surely $f$ is $C^{2+}$. Since $\kappa$ is positive definite, continuous and $\kappa(0)=1$, by Bochner's theorem there exists a probability measure $\rho$ such that
\begin{equation}\label{e:spectral measure}
	\kappa(x)=\int_{\R^2}e^{i\langle t,x\rangle} \, d\rho(t) ;
\end{equation}
this is known as the \emph{spectral measure} of $f$, and must be Hermitian (that is, $\rho(-A)=\rho(A)$ for all Borel sets $A$). Since the distribution of $f$ is uniquely determined by its covariance function (Kolmogorov's theorem), \eqref{e:spectral measure} shows that the distribution of $f$ is uniquely determined by $\rho$.

The geometric properties of $\mathcal{N}$ are of interest, in part, because in the case that $f$ is a random eigenfunction of the Laplacian they relate to a significant conjecture in the physics literature: the Berry conjecture \cite{berry1977regular}. A summary of this conjecture and other research on this topic may be found in \cite{wigman2012nodal}. One of the main analytical results concerning this set, due to Nazarov and Sodin (\cite{SodinNazarov2015asymptotic} and \cite{sodin2016lectures}), states that the number of components of $\mathcal{N}$ in a large domain scales like the area of the domain. Specifically, if $N_R$ denotes the number of components of $\mathcal{N}$ inside the centred ball of radius $R >0$, then provided $f$ is ergodic, there exists a constant $c_{LS}=c_{LS}(\rho)\geq 0$ such that
\begin{displaymath}
	\frac{N_R}{\pi R^2}\rightarrow c_{LS}
\end{displaymath}
as $R\to\infty$, where convergence occurs almost surely and in $L^1$. Nazarov-Sodin also obtained analogous results in higher dimensions and for Gaussian ensembles on manifolds \cite{SodinNazarov2015asymptotic}. In the case that $f$ is not ergodic, it has been shown \cite{kurlberg2017variation} (under the additional assumption that $\rho$ has compact support) that the expected number of nodal components, scaled by the area, still converges, i.e.\
\begin{displaymath}
	\frac{\mathbb{E}[N_R]}{\pi R^2}\rightarrow c_{LS}
\end{displaymath}
as $R\to\infty$. Further, in \cite{kurlberg2017variation} it was also shown that among fields with compactly supported spectral measures, the constant $c_{LS}$ varies continuously with $\rho$ (in the weak-$*$ topology).

\subsection{The main results}
The first contribution of this paper is to extend the results of Nazarov-Sodin and \cite{kurlberg2017variation} to arbitrary levels. For $u \in \mathbb{R}^2$ and $R > 0$ let $B(u,R)$ be the ball of radius $R$ centred at $u$ and $B(R):=B(0,R)$. Let $\mathcal{N}_\ell=\{x\in\R^2:f(x)=\ell\}$ denote a level set of $f$ and let $N_{LS,R}(\ell)$ be the number of components of $\mathcal{N}_\ell$ contained in $B(R)$ (i.e.\ those which intersect $B(R)$ but not $\partial B(R)$). We consider fields satisfying the following assumptions: 
\begin{condition}\label{Conditions 1}
	The Gaussian field $f$ satisfies:
	\begin{enumerate}
		\item For some $\nu>0$, $f \in C^{2+\nu}_\text{loc}(\R^2)$ almost surely;
		\item $\nabla^2 f(0)$ is a non-degenerate Gaussian vector (here, and later on, we treat $\nabla^2 f$ as a vector of distinct partial derivatives $(f_{xx},f_{xy},f_{yy})$);
		\item For any $t\in\R^2$, if $f(t)-f(0)$ is a non-degenerate Gaussian variable then the Gaussian vector $(f(t)-f(0),\nabla f(t),\nabla f(0))$ is non-degenerate.
	\end{enumerate} 
\end{condition}
We note that these assumptions are quite minimal. The first condition holds if the covariance function $\kappa$ is $C^{4+\epsilon}$ for some $\epsilon>2\nu$, or equivalently when $\int_{\R^2}\lvert\lambda\rvert^{4+\epsilon}\;d\rho(\lambda)<\infty$. The second condition is equivalent to the support of $\rho$ not being contained in the union of two lines through the origin. The third condition holds provided that the support of $\rho$ is not too degenerate; in particular it holds if the support of $\rho$ contains an open set.

\begin{theorem}\label{t:main level}
	Let $f$ satisfy Conditions \ref{Conditions 1}. For each $\ell\in\R$, there exists $c_{LS}(\rho,\ell)\geq 0$ such that
	\begin{displaymath}
		\mathbb{E}[N_{LS,R}(\ell)]=c_{LS}(\rho,\ell)\cdot\pi R^2+O\left(R\right)
	\end{displaymath}
	as $R\to\infty$. The constant implied by the $O(\cdot)$ notation may depend on $\rho$ but is independent of $\ell$. If $f$ is also ergodic, then
	\begin{displaymath}
		\frac{N_{LS,R}(\ell)}{\pi R^2}\rightarrow c_{LS}(\rho,\ell)
	\end{displaymath}
	almost surely and in $L^1$.
\end{theorem}
There is also interest in studying the number of excursion sets of Gaussian fields. Let $N_{ES,R}(\ell)$ denote the number of components of $\{x\in\R^2:f(x)\geq\ell\}$ contained in $B(R)$.
\begin{samepage}
\begin{theorem}\label{t:main excursion}
	Let $f$ satisfy Conditions \ref{Conditions 1}. For each $\ell\in\R$, there exists $c_{ES}(\rho,\ell)\geq 0$ such that
	\begin{displaymath}
		\mathbb{E}[N_{ES,R}(\ell)]=c_{ES}(\rho,\ell)\cdot\pi R^2+O\left(R\right)
	\end{displaymath}
	as $R\to\infty$. The constant implied by the $O(\cdot)$ notation may depend on $\rho$ but is independent of $\ell$. If $f$ is also ergodic, then
	\begin{displaymath}
		\frac{N_{ES,R}(\ell)}{\pi R^2}\rightarrow c_{ES}(\rho,\ell)
	\end{displaymath}
	almost surely and in $L^1$.
\end{theorem}
\end{samepage}
\begin{remark}
	Our use of the domain $B(R)$ in the definitions of $N_{LS,R}(\ell)$ and $N_{ES,R}(\ell)$ is mainly for simplicity, and after minor modifications the proofs of Theorems~\ref{t:main level} and \ref{t:main excursion} go through equally well for rescaled copies of any bounded reference domain $\Omega \subset \mathbb{R}^2$, provided that $\Omega$ is convex and $\partial\Omega$ is piecewise smooth (and probably more generally as well). Moreover the limiting constants $c_{LS}$ and $c_{ES}$ do not depend on $\Omega$ (after replacing the scaling factor $\pi R^2$ by $\text{Area}(\Omega) R^2$).
\end{remark}

\begin{remark}
	Theorem~\ref{t:main excursion} can also be applied to lower excursion sets (the components of $\{x\in\R^2:f(x)\leq\ell\}$) since $\{x\in\R^2:f(x)\leq\ell\}=\{x\in\R^2:-f(x)\geq-\ell\}$ and $f$ has symmetric distribution. Theorem~\ref{t:main excursion} can then be used to prove Theorem~\ref{t:main level} by making use of Euler's formula to show that the number of level set components $N_{LS,R}(\ell)$ is equal to the number of upper and lower excursion set components in $B(R)$ and a bounded error term (see the proof of Lemma~\ref{l:main topological} for details of this argument).
\end{remark}

The symmetry of $f$ along with the observations in the previous remark immediately give the following corollary.
\begin{corollary}\label{c:basic}
	Let $f$ satisfy Conditions \ref{Conditions 1}. Then
	\begin{enumerate}[(1)]
		\item $c_{LS}(\ell)=c_{LS}(-\ell)$ for all $\ell\in\R$,
		\item $c_{LS}(\ell)=c_{ES}(\ell)+c_{ES}(-\ell)$ for all $\ell\in\R$,
		\item $c_{LS}(0)=2c_{ES}(0)$.
	\end{enumerate}
\end{corollary}

Theorems~\ref{t:main level} and~\ref{t:main excursion} are, in isolation, only a modest improvement on previous results; they could be proven by slightly adapting the analysis in \cite{SodinNazarov2015asymptotic} and \cite{kurlberg2017variation}. The main novelty of our work is to relate the functionals $c_{LS}$ and $c_{ES}$ to the density of critical points of $f$ of different types and at different levels. To state the relationship, we shall require, in particular, a classification of the saddle points of the field into two types:
\begin{definition}\label{d:lower connected aperiodic}
	Let $x_0$ be a saddle point of a function $g:\R^2\to\R$ such that there are no other critical points at the same level as $x_0$. We say that $x_0$ is an \emph{upper connected saddle} if it is in the closure of only one component of $\{x\in\R^2:g(x)>g(x_0)\}$. Similarly, $x_0$ is said to be \emph{lower connected} if it is in the closure of only one component of $\{x\in\R^2:g(x)<g(x_0)\}$.
\end{definition}
We say that a Gaussian field $f$ satisfying Conditions \ref{Conditions 1} is \textit{periodic} if there exists $x\neq 0$ with $\kappa(x)=1$ and is \textit{aperiodic} otherwise. A Gaussian field which is aperiodic almost surely has no two critical points at the same level (see Lemma~\ref{Gaussian fields main theorem}), and we will show that for such fields, all saddle points of $f$ satisfy exactly one of the two conditions in Definition~\ref{d:lower connected aperiodic}. In the case that $f$ is periodic, we will require a more general definition for classifying saddle points as upper or lower connected, which is given in Section~\ref{s:critical points}.

Previous work has shown that, for sufficiently regular isotropic\footnote{A Gaussian field is said to be \textit{isotropic} if its covariance function $\kappa(x)$ can be expressed as a function of $\lvert x\rvert$ where $\lvert \cdot\rvert$ denotes the Euclidean norm.} Gaussian fields, the expected number of local maxima, local minima or saddle points with value in a certain interval can be expressed as the integral of an explicit density function \cite{cheng2015expected,Dennis2007}. In Section~\ref{s:critical points} we prove the following version of this result, which applies to more general Gaussian fields and also isolates the upper and lower connected saddles, but does not explicitly identify the densities. This proposition uses Definition~\ref{d:lower connected general} for upper and lower connected saddle points (which coincides with Definition~\ref{d:lower connected aperiodic} for aperiodic fields).

\begin{proposition}\label{p:density existence}
	Let $f$ satisfy Conditions \ref{Conditions 1}. Then there exist non-negative functions $p_{m^+}$, $p_{m^-}$, $p_{s^+}$, $p_{s^-}$ and $p_s$ on $\R$ such that the following holds. Let $\Omega\subset\R^2$ be compact and $\partial\Omega$ have finite Hausdorff-1 measure. Let $\ell \in \mathbb{R}$ and let $N_{m^+}(\ell)$, $N_{m^-}(\ell)$, $N_{s^+}(\ell)$, $N_{s^-}(\ell)$ and $N_s(\ell)$ denote the number of local maxima, local minima, upper connected saddles, lower connected saddles and saddles of $f$ in $\Omega$ with level above $\ell$ respectively. Then
	\begin{displaymath}
		\mathbb{E}[N_h(\ell)]=\emph{Area}(\Omega)\int_\ell^\infty p_h(x) \, dx
	\end{displaymath}
	for $h=m^+,m^-,s^+,s^-,s$. Furthermore, these functions can be chosen to satisfy the relations $p_{m^+}(x) = p_{m^-}(-x)$, $p_{s^+}(x) = p_{s^-}(-x)$ and $p_{s^-}+p_{s^+}=p_s$, and such that $p_{m^+}$, $p_{m^-}$ and $p_s$ are continuous.\end{proposition}

The main theorem of the paper gives an explicit expression for the functionals $c_{LS}$ and $c_{ES}$ in terms of the densities introduced in Proposition~\ref{p:density existence}. As a result, we deduce the absolute continuity of these functionals as the level varies. In the case of spectral measures with compact support, we also show the joint continuity of these functionals with respect to both the level and the spectral measure.

For each $f$ satisfying Conditions \ref{Conditions 1} we can associate a spectral measure $\rho$ via \eqref{e:spectral measure}. Let $\mathcal{P}_c$ denote the collection of such measures that are supported in the closure of $B(1)$. By rescaling the axes, the results we state for $\mathcal{P}_c$ can be shown to hold for any spectral measures with compact support.
\begin{theorem}\label{t:integral equality}
	Let $f$ satisfy Conditions \ref{Conditions 1} and let $p_{m^+}$, $p_{m^-}$, $p_{s^+}$, $p_{s^-}$ denote the densities specified in Proposition~\ref{p:density existence}. Then
	\begin{equation}
		\label{e:integral equality1}
		c_{LS}(\rho,\ell)=\int_\ell^\infty p_{m^+}(x)-p_{s^-}(x)+p_{s^+}(x)-p_{m^-}(x) \, dx
	\end{equation}
	and
	\begin{equation}
		\label{e:integral equality2}
		c_{ES}(\rho,\ell)=\int_\ell^\infty p_{m^+}(x)-p_{s^-}(x) \, dx
	\end{equation}
	and hence $c_{LS}$ and $c_{ES}$ are absolutely continuous in $\ell$. In addition $c_{LS}$ and $c_{ES}$ are jointly continuous in $(\rho,\ell)\in\mathcal{P}_c\times\R$ where $\mathcal{P}_c$ is given the weak-$*$ topology.
\end{theorem}

\begin{remark}
	Theorem~\ref{t:integral equality} provides a new tool with which to analyse the Nazarov-Sodin constant. Since the densities $p_{m^+}$, $p_{m^-}$ and $p_s$ are in principle known by the Kac-Rice formula, our result demonstrates that the study of the Nazarov-Sodin constant can be reduced to an analysis of the density $p_{s^-}$ (or, equivalently, $p_{s^+}=p_s-p_{s^-}$), which may be an easier quantity to handle.
\end{remark}

\begin{remark}
	Theorems~\ref{t:main level},~\ref{t:main excursion} and~\ref{t:integral equality} can be generalised to many examples of non-Gaussian stationary random fields, since our proof requires only that the field satisfies certain topological properties almost surely and is sufficiently regular to apply the Kac-Rice formula (see the more general version of our main results stated in Propositions~\ref{NS integral equality} and~\ref{ES integral equality} below).
	
	We also believe Theorems~\ref{t:main level},~\ref{t:main excursion} and~\ref{t:integral equality} could be generalised to higher dimensions, with the analogues of \eqref{e:integral equality1} and \eqref{e:integral equality2} still valid once the saddle points defining $p_{s^+}$ and $p_{s^-}$ are replaced with critical points of index $1$ and $d-1$ respectively (using a more general definition for upper and lower connected saddles). However, since some of the topological arguments in the proof increase in complexity in higher dimensions, in the interest of simplicity we do not pursue this generalisation here.
\end{remark}

Let us mention the key intuition behind the proof of Theorem~\ref{t:integral equality}. This theorem is based on a deterministic relationship between the excursion sets and critical points of sufficiently regular planar functions which is closely related to Morse theory. The excursion set of such a function above a level deforms continuously as the level increases, provided it does not pass through a critical point. In particular, there is no change in the number of components of the excursion set. When passing through a critical point, the topology of the excursion set changes in a way that is predicted by Morse theory and depends on the index of the critical point. For local maxima and local minima, the number of components of the excursion set changes in a consistent way. For saddle points, the change in the number of components is determined by whether it is upper connected or lower connected (see Figure~\ref{Fig_1+2}). Ultimately, Theorem~\ref{t:integral equality} is a probabilisitic expression, in the setting of Gaussian fields, of this deterministic relationship between excursion set components and critical points of various types.
\begin{figure*}[h!]
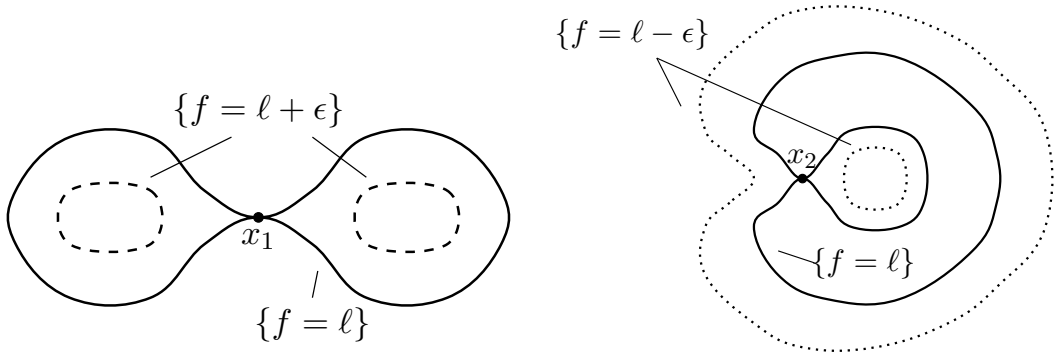

	\centering
	\subfloat{\input{Fig_1.tikz}}
	\subfloat{\input{Fig_2.tikz}}
	\caption{The number of excursion set components increases by one on passing through the lower connected saddle $x_1$ and is constant on passing through the upper connected saddle $x_2$.}\label{Fig_1+2}
\end{figure*}

We also briefly discuss the assumptions required for our results (i.e.\ Conditions \ref{Conditions 1}). The assumption that $f$ is almost surely $C^{2+\nu}$ is necessary to apply the topological arguments that we borrow from Morse theory. The non-degeneracy assumptions on $\nabla^2 f(0)$ and $(f(t)-f(0),\nabla f(t),\nabla f(0))$ are used with the Kac-Rice theorem (see Section~\ref{s:proof main results}) to show certain non-degeneracy properties of $f$, which are again necessary for our topological arguments.

It is natural, therefore, to ask whether our results still apply when these assumptions are weakened. While we suspect that this is true, we are not able to show it with our methods. On the other hand, for a certain special case of non-trivial degenerate field -- namely, the case where the spectral measure $\rho$ is supported on at most five points -- we are able to give a complete description of $c_{LS}$ and $c_{ES}$, which in particular shows that the main results hold also in this case (see Section~\ref{s:degen} below).

\subsection{Bounds on $c_{LS}$ and $c_{ES}$}
\label{s:bounds}
It is possible to bound the expected number of excursion sets or level sets of stationary Gaussian fields using local estimates (i.e.\ estimates which depend only on the derivatives of $\kappa$ at the origin). Here we outline how these estimates apply to the functionals $c_{ES}$ and $c_{LS}$ and how they can be better characterised by making use of our results.

\begin{corollary}\label{c:isotropic bounds 1}
	Let $f$ be a Gaussian field satisfying Conditions \ref{Conditions 1} with covariance function $\kappa$. Then for $\ell\in\R$
	\begin{equation}\label{e:c_ES difference}
		c_{ES}(\ell)-c_{ES}(-\ell)=\sqrt{\det\nabla^2\kappa(0)}\:\frac{\ell}{2\pi}\: \phi(\ell)
	\end{equation}
	where $\phi$ denotes the standard normal probability density function.
\end{corollary}

\begin{remark}
	Since $-\nabla^2\kappa(0)$ is the covariance matrix of $\nabla f(0)$, this result implies that the difference $c_{ES}(\ell)-c_{ES}(-\ell)$ depends only on the covariance of $\nabla f(0)$ and not on the covariance of higher order derivatives of
	$f$.
\end{remark}

\begin{remark}
	This result was proven for isotropic fields in \cite{swerling1962statistical} using a winding number calculation. We prove this result using Theorem \ref{t:integral equality} which simplifies the calculation in the non-isotropic case and also highlights an interesting identity. Specifically, the proof is as follows: substituting~\eqref{e:integral equality2} into the left hand side of~\eqref{e:c_ES difference} and using the symmetries $p_{m^+}(x) = p_{m^-}(-x)$, $p_{s^+}(x) = p_{s^-}(-x)$ and $p_{s^-}+p_{s^+}=p_s$, we see that
	\begin{align*}
		c_{ES}(\ell)-c_{ES}(-\ell)&=\int_\ell^\infty p_{m^+}(x)-p_s(x)+p_{m^-}(x)\;dx\\
		&=\mathbb{E}(N_{m^+}(\ell)-N_s(\ell)+N_{m^-}(\ell))
	\end{align*}
	where $N_{m^+}(\ell)$, $N_s(\ell)$ and $N_{m^-}(\ell)$ are the number of critical points above level $\ell$ as defined in Proposition \ref{p:density existence} for $\Omega=B(1/\sqrt{\pi})$. Lemma 11.7.1 of \cite{RFG} states that this alternating sum is precisely the right hand side of~\eqref{e:c_ES difference}.
	
	The alternating sum of critical points of a function of different indices above a certain level can be used to calculate the Euler characteristic of the excursion set of the function above the level (see Chapter 9 of \cite{RFG}). When working on finite subsets of the plane, boundary effects must be considered, but these become negligible as the area of the subset increases. Formally, if we let $\varphi(A)$ denote the Euler characteristic of a set $A$, then Theorem 11.7.2 of \cite{RFG} gives an expression for the expected Euler characteristic of an excursion set on a cube (including boundary effects) which implies that
	\begin{displaymath}
		c_{ES}(\ell)-c_{ES}(-\ell)=\lim_{R\to\infty}\frac{1}{R^2}\mathbb{E}\left(\varphi\left(\{x\in[0,R]^2 : f(x)\geq\ell\}\right)\right)
	\end{displaymath}
\end{remark}

Corollary~\ref{c:isotropic bounds 1} immediately gives a lower bound for $c_{ES}(\ell)$ since $c_{ES}(-\ell)\geq 0$. We now consider an upper bound on $c_{LS}$. This is most easily formulated when $f$ is an isotropic Gaussian field. In this case its covariance function may be expressed as $\kappa(x)=K(\lvert x\rvert)$. It is shown in \cite{cheng2015expected} that 
\begin{equation}\label{e:isotropic parameters}
	\lambda:=\frac{-\sqrt{3}K^{(2)}(0)}{\sqrt{K^{(4)}(0)}} \in \left(0, \sqrt{2}\right] \quad \text{and} \quad  \eta^2:=\frac{-6K^{(2)}(0)}{K^{(4)}(0)} \in [0, \infty)
\end{equation}
parameterise the critical point densities $p_{m^+}$, $p_{m^-}$ and $p_s$ of isotropic fields in any dimension. It can also be shown that, for planar isotropic fields, if $\lambda=\sqrt{2}$ then $\kappa(x)=J_0(\sqrt{8/\eta^2}\:\lvert x\rvert)$, where $J_0$ is the $0$-th Bessel function. When $\eta^2=8$, this particular field is known as the random plane wave (hereafter abbreviated to RPW), and is an object of great interest as it is the subject of the Berry conjecture \cite{wigman2012nodal}.

\begin{proposition}[\cite{swerling1962statistical}]\label{p:isotropic bounds 2}
	Let $f$ be an isotropic Gaussian field satisfying Conditions \ref{Conditions 1} with covariance function $\kappa(x)=K(\lvert x\rvert)$. Then for $\ell\geq 0$
	\begin{equation}\label{e:c_NS upper bound}
		c_{LS}(\ell)\leq \frac{\lambda^2}{\pi\eta^2}\:\phi(\ell)\left(\frac{2\sqrt{3-\lambda^2}}{\lambda}\:\phi\left(\lambda\ell/\sqrt{3-\lambda^2}\right)+\ell\left(2\Phi\left(\lambda\ell/\sqrt{3-\lambda^2}\right)-1\right)\right) .
	\end{equation}
	where $\phi$ and $\Phi$ denote the standard normal probability density and cumulative density functions respectively.
	
	In particular, if $\lambda^2>\frac{6e}{2e+\pi}\approx 1.9$ then by Corollary \ref{c:isotropic bounds 1}, $c_{LS}(0)<c_{LS}(1)$ so that $c_{LS}$ is at least bimodal (that is, it has at least two local maxima). 
\end{proposition}

\begin{remark}
	This result is proven in \cite{swerling1962statistical} using the method of `flip points'. Specifically, for any fixed direction $u$, the number of level set components at level $\ell$ in a finite region is bounded above by half the number of points $t$ in the region such that $f(t)=\ell$ and $\partial_u f(t)=0$ where $\partial_u$ denotes the partial derivative in the direction $u$. The expected number of such points can be computed using the Kac-Rice formula. For isotropic fields, the choice of direction $u$ is irrelevant. This method could also be applied to non-isotropic fields, and the bound could be optimised over the direction $u$, but we omit this for simplicity.
	
	It is also possible to construct an upper bound on $c_{LS}$ using the inequality $0\leq p_{s^-}\leq p_s$ and the densities $p_{m^+}$, $p_s$ and $p_{m^-}$, which are explicitly known for isotropic fields. However this bound is larger than the bound in~\eqref{e:c_NS upper bound} at all levels $\ell$.
\end{remark}

\begin{remark}
	Figure~\ref{Fig_3+4} shows lower and upper bounds for $c_{ES}$ and $c_{LS}$ for the RPW based on~\eqref{e:c_ES difference}, ~\eqref{e:c_NS upper bound} and the equality $c_{LS}(\ell)=c_{ES}(\ell)+c_{ES}(-\ell)$ (recall that $\lambda = \sqrt{2}$ and $\eta^2 = 8$ in this case). Although these bounds are not particularly tight for $\ell$ close to $0$, they quickly become accurate as $\lvert\ell\rvert$ increases. For $\ell\geq 1$ both the upper and lower bounds on $c_{LS}(\ell)$ are within 5.1\% of the true value (since the upper bound is within 5.1\% of the lower bound) while for $\ell\geq 1.5$ both bounds are within 0.6\% of the true value.
	
	For fields with lower values of $\lambda$, the percentage difference between the upper and lower bounds on $c_{LS}(\ell)$ is a bit larger. Figure~\ref{Fig_5+6} shows the corresponding bounds on $c_{ES}$ and $c_{LS}$ for the Bargmann-Fock field: the centred, planar Gaussian field with covariance kernel $\kappa(x)=\exp(-\lvert x\rvert^2/2)$, for which $\lambda=1$ and $\eta^2=2$. In this case for $\ell\geq 1$ the upper and lower bounds on $c_{LS}(\ell)$ are within 40\% of the true value, and the corresponding accuracies for $\ell\geq 1.5$, $\ell\geq 2$ and $\ell\geq 2.5$ are 14\%, 5\% and 2\% respectively.
\end{remark}

\begin{figure*}[h!]
	\centering
	\subfloat[\label{Fig_3}]{\includegraphics[scale=0.5]{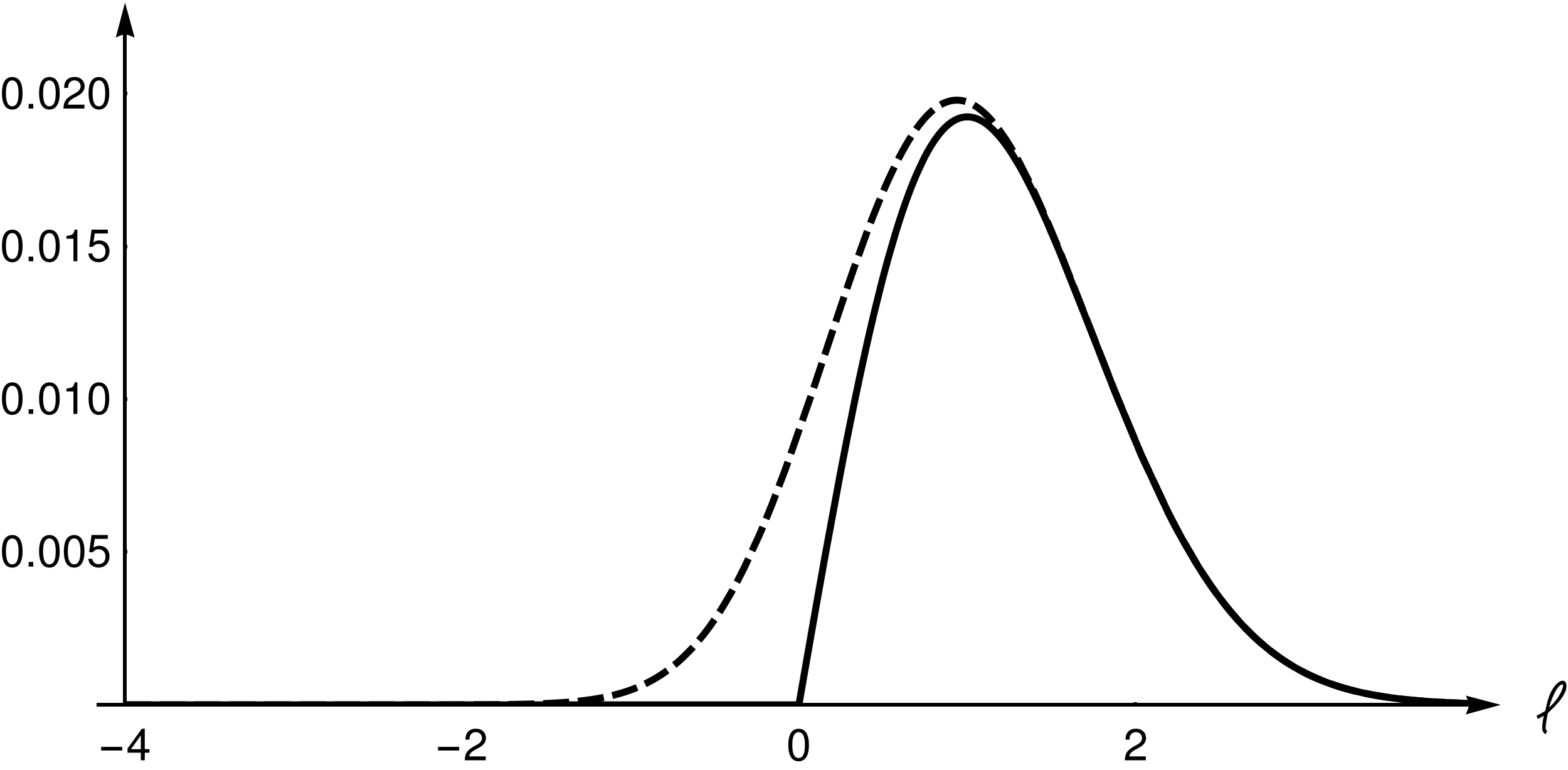}}\qquad
	\subfloat[\label{Fig_4}]{\includegraphics[scale=0.5]{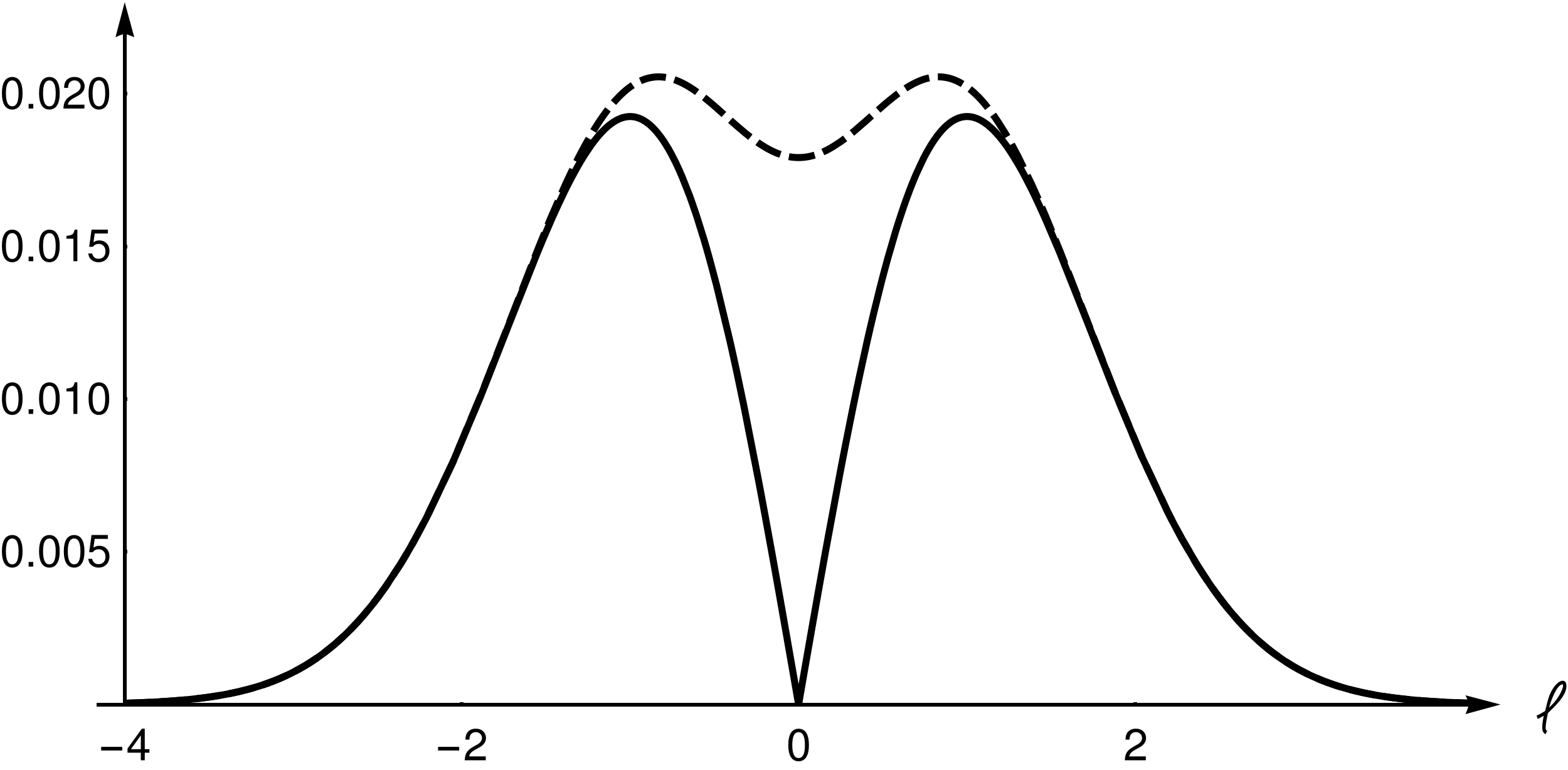}}
	\caption{Subfigures~\ref{Fig_3} and~\ref{Fig_4} show lower bounds (solid) and upper bounds (dashed) for $c_{ES}(\rho,\ell)$ and $c_{LS}(\rho,\ell)$ respectively in the RPW case, for which $\lambda = \sqrt{2}$ and $\eta^2 = 8$.}\label{Fig_3+4}
\end{figure*}

\begin{figure*}[h!]
	\centering
	\subfloat[\label{Fig_5}]{\includegraphics[scale=0.5]{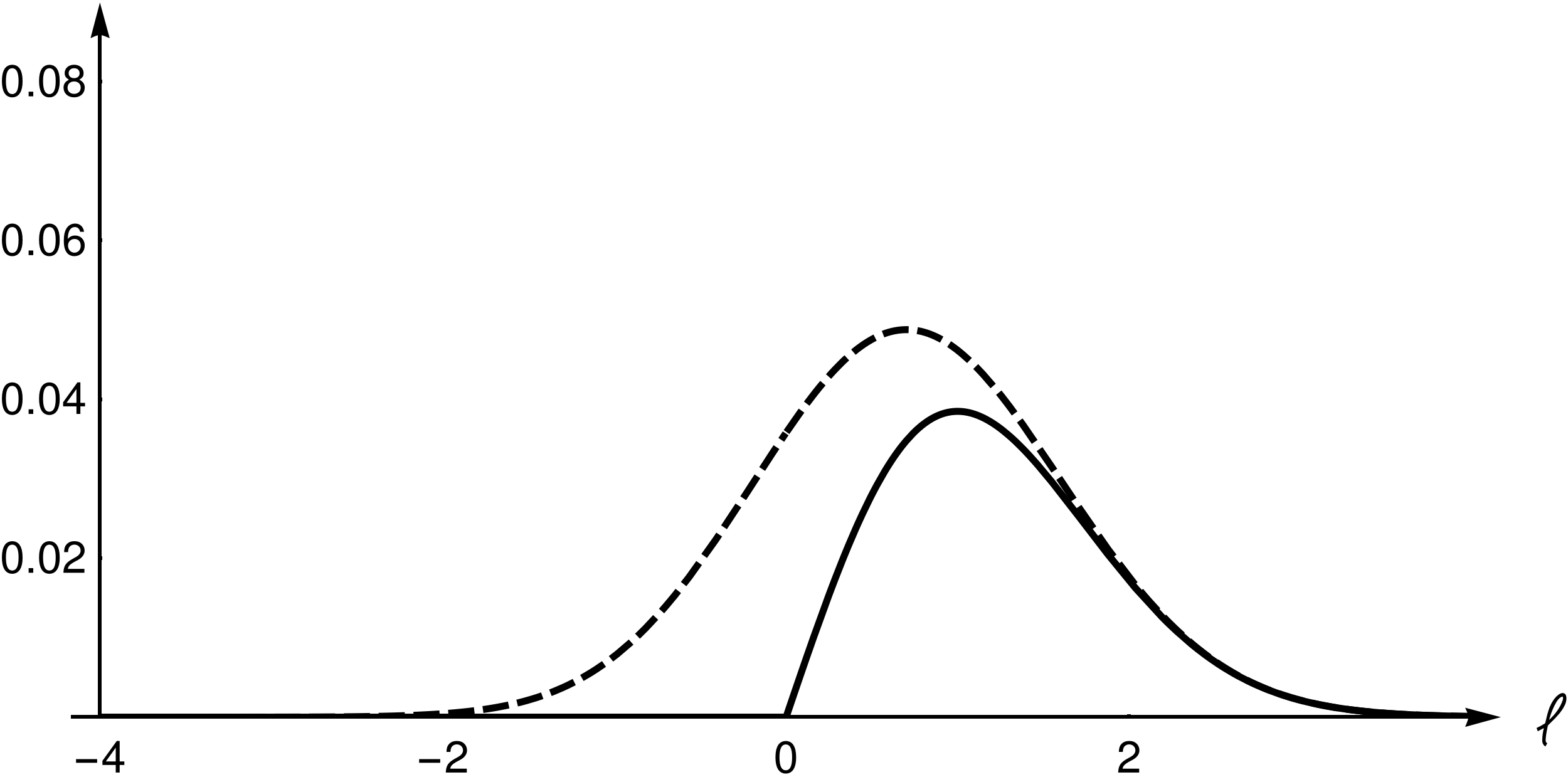}}\qquad
	\subfloat[\label{Fig_6}]{\includegraphics[scale=0.5]{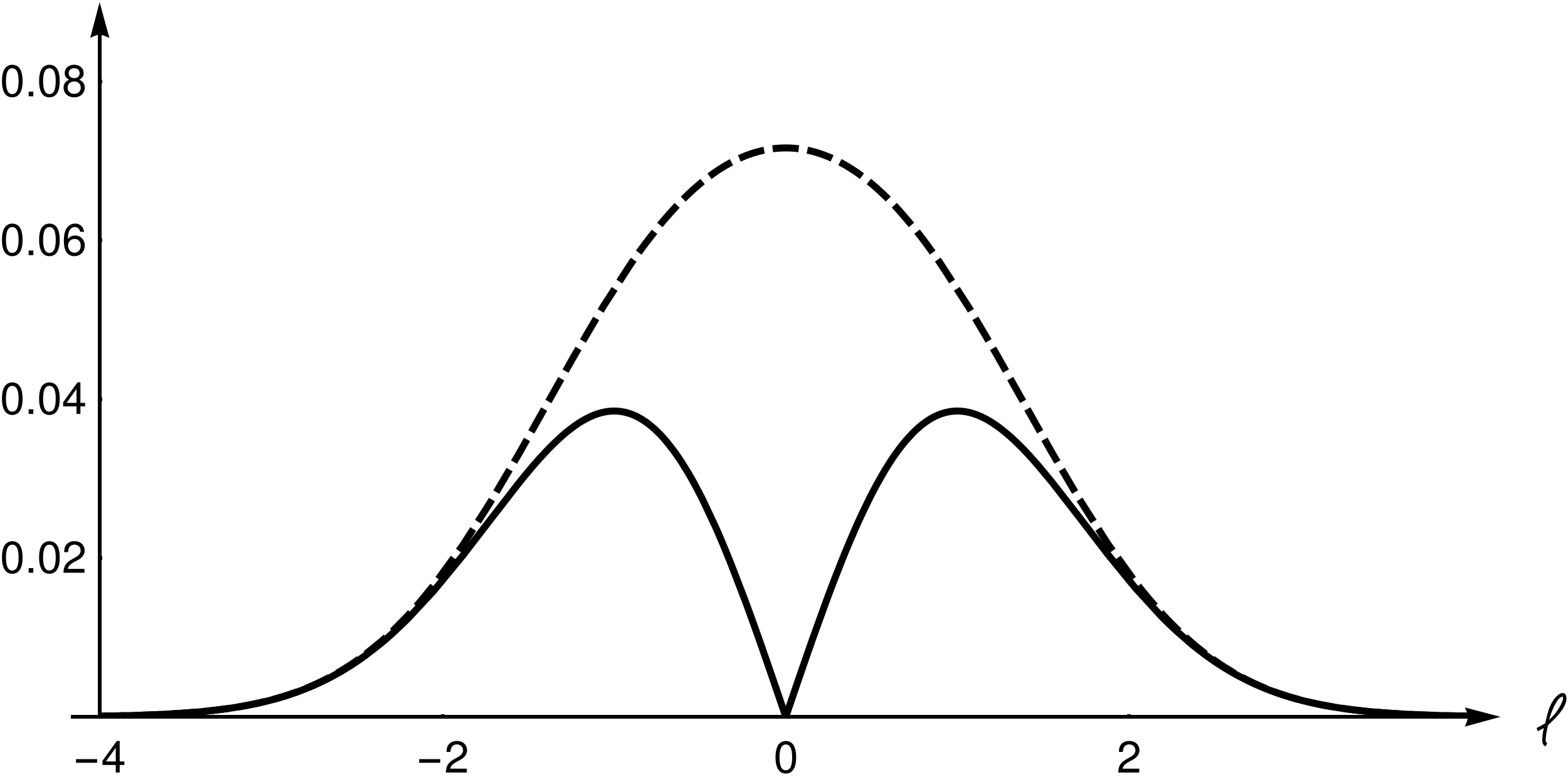}}
	\caption{Subfigures~\ref{Fig_5} and~\ref{Fig_6} show lower bounds (solid) and upper bounds (dashed) for $c_{ES}(\rho,\ell)$ and $c_{LS}(\rho,\ell)$ respectively, where $\rho$ is the density of a standard two-dimensional Gaussian random vector (i.e.\ the spectral measure of the Bargmann-Fock field), for which $\lambda=1$ and $\eta^2=2$.}\label{Fig_5+6}
\end{figure*}

\begin{remark}
	Corollary~\ref{c:isotropic bounds 1} states that $c_{ES}(\rho,\ell)>0$ for all $\ell>0$, and so $c_{LS}(\rho,\ell)>0$ for all $\ell\neq 0$, however the positivity of $c_{LS}(\rho,0)$ is not addressed by this method. In \cite{SodinNazarov2015asymptotic} Nazarov and Sodin provide sufficient conditions for the positivity of $c_{LS}(\rho,0)$ in terms of the spectral measure which cover almost all non-degenerate cases.
	
	The only current quantitative lower bound for $c_{LS}(\rho,0)$ has been proven for the RPW.  In this case it has been shown that $c_{LS}(0)\geq 1.1\times 10^{-5}$ (see \cite{ingremeau2018lower}) although numerical simulations suggest that $c_{LS}(0)\approx 0.0589/(4\pi)\approx 0.00469$ (see \cite{beliaev2013bogomolny} and references therein).
\end{remark}

Although the bounds in the previous two corollaries can be proven by local methods (i.e.\ without using our analysis of saddle points) there is no scope for improving these estimates using the same methods. However it may be possible to get tighter bounds on $c_{LS}$ and $c_{ES}$ through Theorem~\ref{t:integral equality} and a better characterisation of $p_{s^-}$.

Using the explicitly known densities $p_{m^+}$, $p_{m^-}$ and $p_s$ along with Theorem~\ref{t:integral equality} allows us to derive some monotonicity properties of $c_{ES}$ and $c_{LS}$.

\begin{corollary}\label{c:monotonicity}
	Let $f$ satisfy Conditions \ref{Conditions 1} and $\lambda$ be defined as above. Then $c_{ES}(\rho,\ell)$ and $c_{LS}(\rho,\ell)$ are strictly decreasing in $\ell$ for $\ell>\sqrt{2}/\lambda$, so that any local maxima of $c_{LS}$ must be contained in $[-\sqrt{2}/\lambda,\sqrt{2}/\lambda]$ and any local maxima of $c_{ES}$ must be contained in $(-\infty,\sqrt{2}/\lambda]$.
	
	If $\lambda=\sqrt{2}$, then $c_{ES}(\ell)$ is non-decreasing on $(-\infty,0]$ and strictly decreasing on $[1,\infty)$, so that any strict local maxima of $c_{ES}$ must be contained in $[0,1]$.
\end{corollary}

\subsection{A special class of degenerate fields}
\label{s:degen}
As mentioned above, our methods do not cover the case in which $\rho$ is supported on the union of two lines through the origin. In this section we consider a certain special class of such fields for which we can give a more or less complete description of $c_{LS}$ and $c_{ES}$.

If the spectral measure $\rho$ is supported on a single line through the origin, it is well known that $f$ is almost surely constant in one direction and so in particular has no compact level domains. In this case, all of our results hold trivially with $c_{LS}$, $c_{ES}$ and the critical point densities identically equal to zero. The simplest non-trivial case of degenerate $\rho$ are spectral measures which are supported on four or five points. In this case, we can compute $c_{ES}$ and $c_{LS}$ explicitly.

Recall that a random variable $Y$ is said to be Rayleigh distributed with parameter $\sigma^2>0$, denoted $Y\sim\text{Ray}(\sigma)$, if $\mathbb{P}(Y\leq x)=1-e^{-x^2/(2\sigma^2)}$ for all $x\geq 0$.
\begin{proposition}\label{p:cilleruello c_NS}
	Let $f$ be the Gaussian field with spectral measure
	\begin{displaymath}
		\rho=\alpha\delta_0+\frac{\beta}{2}(\delta_K+\delta_{-K})+\frac{\gamma}{2}(\delta_L+\delta_{-L})
	\end{displaymath}
	where $\beta,\gamma>0$, $\alpha=1-\beta-\gamma\geq 0$ and $K,L\in\R^2$ are linearly independent. Then
	\begin{displaymath}
		\mathbb{E}[N_{ES,R}(\ell)]=c_{ES}(\ell)\cdot\pi R^2+O(R)
	\end{displaymath}
	and
	\begin{displaymath}
		\mathbb{E}[N_{LS,R}(\ell)]=c_{LS}(\ell)\cdot\pi R^2+O(R)
	\end{displaymath}
	where
	\begin{equation*}
		\begin{aligned}
			c_{ES}(\ell)&=\lvert K\times L\rvert\cdot\mathbb{P}\left(\lvert Y_1-Y_2\rvert\leq \ell+X_0\leq Y_1+Y_2\right),\\
			c_{LS}(\ell)&=\lvert K\times L\rvert\cdot\mathbb{P}\left(\lvert Y_1-Y_2\rvert\leq\lvert\ell+X_0\rvert\leq Y_1+Y_2\right),
		\end{aligned}
	\end{equation*}
	$\times$ denotes the cross product, $X_0\sim\mathcal{N}(0,\alpha)$, $Y_1\sim\textup{Ray}(\sqrt{\beta})$, $Y_2\sim\textup{Ray}(\sqrt{\gamma})$ and $X_0,Y_1,Y_2$ are independent. Moreover the constants implied by the $O(\cdot)$ notation in these expressions are independent of $\ell$.
	
	So in particular, $c_{LS}(\ell)=0$ if and only if $\ell=\alpha=0$. If $c_{LS}(\ell)\neq 0$ then $N_{LS,R}(\ell)/(\pi R^2)$ converges in $L^1$ to a non-constant random variable and hence does not converge a.s.\ to a constant, and this statement also holds for $c_{ES}$ and $N_{ES,R}(\ell)/(\pi R^2)$.
	
	Furthermore there exist functions $p_{m^+}$, $p_{s^-}$, $p_{m^-}$, $p_{s^+}$ and $p_s$ satisfying the conclusions of Proposition~\ref{p:density existence}, and these are defined by
	\begin{align*}
		p_{m^+}(x)&=p_{m^-}(-x)=\lvert K\times L\rvert\cdot p_{X_0+Y_1+Y_2}(x)\\
		p_{s^-}(x)&=p_{s^+}(-x)=\lvert K\times L\rvert\cdot p_{X_0+\lvert Y_1-Y_2\rvert}(x)
	\end{align*}
	where $p_Z$ denotes the probability density of a random variable $Z$. Therefore the equalities in the conclusion of Theorem~\ref{t:integral equality} hold for $f$.
\end{proposition}

\begin{remark} A stationary centred continuous Gaussian field is ergodic if and only if its spectral measure has no atoms  \cite[Section 6.1]{nazarov2009number}, and so the fields considered in this proposition are not ergodic. The fact that $N_{LS,R}(\ell)/(\pi R^2)$ does not generally converge to $c_{LS}(\ell)$ for these fields shows that the ergodicity requirement in the second part of Theorems~\ref{t:main level} and~\ref{t:main excursion} cannot be entirely relaxed.
\end{remark}

\begin{remark} It has previously been shown that if $\text{spt}(\rho)$ has precisely four points then $c_{LS}(0)=0$ (see \cite{kurlberg2017variation}). Proposition~\ref{p:cilleruello c_NS} shows that $c_{LS}(0)>0$ when $\text{spt}(\rho)$ has five points, and that $c_{LS}(\ell)>0$ for $\ell\neq 0$ when $\text{spt}(\rho)$ has four or five points.
\end{remark}

\begin{remark}
	Figures~\ref{Fig_7+8} and~\ref{Fig_9+10} show $c_{ES}$ and $c_{LS}$ for different values of $\alpha,\beta,\gamma$ when $K=(1,0)$ and $L=(0,1)$. Figure~\ref{Fig_7+8} suggests that $c_{LS}$ is bimodal whenever $\alpha=0$ (we know that it is at least bimodal, since $c_{LS}(0)=0$ and $c_{LS}(\ell)>0$ for $\ell\neq 0$ in this case), however for $\alpha>0$, $c_{LS}$ can be bimodal or unimodal. Combining this observation with Corollary~\ref{c:isotropic bounds 1} raises the interesting question of determining for which Gaussian fields $c_{LS}(\ell)$ is bimodal.
\end{remark}

\begin{remark} Figures~\ref{Fig_7+8} and~\ref{Fig_9+10} also suggest that the derivatives of $c_{ES}$ and $c_{LS}$ need not be continuous at zero when $\alpha=0$. This can be verified analytically since
	\begin{displaymath}
		c_{ES}^\prime(\ell)=p_{s^-}(\ell)-p_{m^+}(\ell)
	\end{displaymath}
	and these densities are explicitly known. In particular, it can be shown that $p_{Y_1+Y_2}$ is everywhere continuous, whereas $p_{\lvert Y_1-Y_2\rvert}$ has a jump discontinuity at zero (but is continuous elsewhere). So $c_{ES}^\prime$ is continuous except at zero. An analogous calculation permits the same conclusion for $c_{LS}^\prime$. On the other hand, when $\alpha\neq 0$, $c_{ES}$ and $c_{LS}$ are continuously differentiable everywhere. A major unresolved question arising from our work is to determine whether $c_{LS}$ and $c_{ES}$ are continuously differentiable more generally (outside degenerate cases).
\end{remark}

\begin{figure*}[h!]
	\centering
	\subfloat{\includegraphics[scale=0.5]{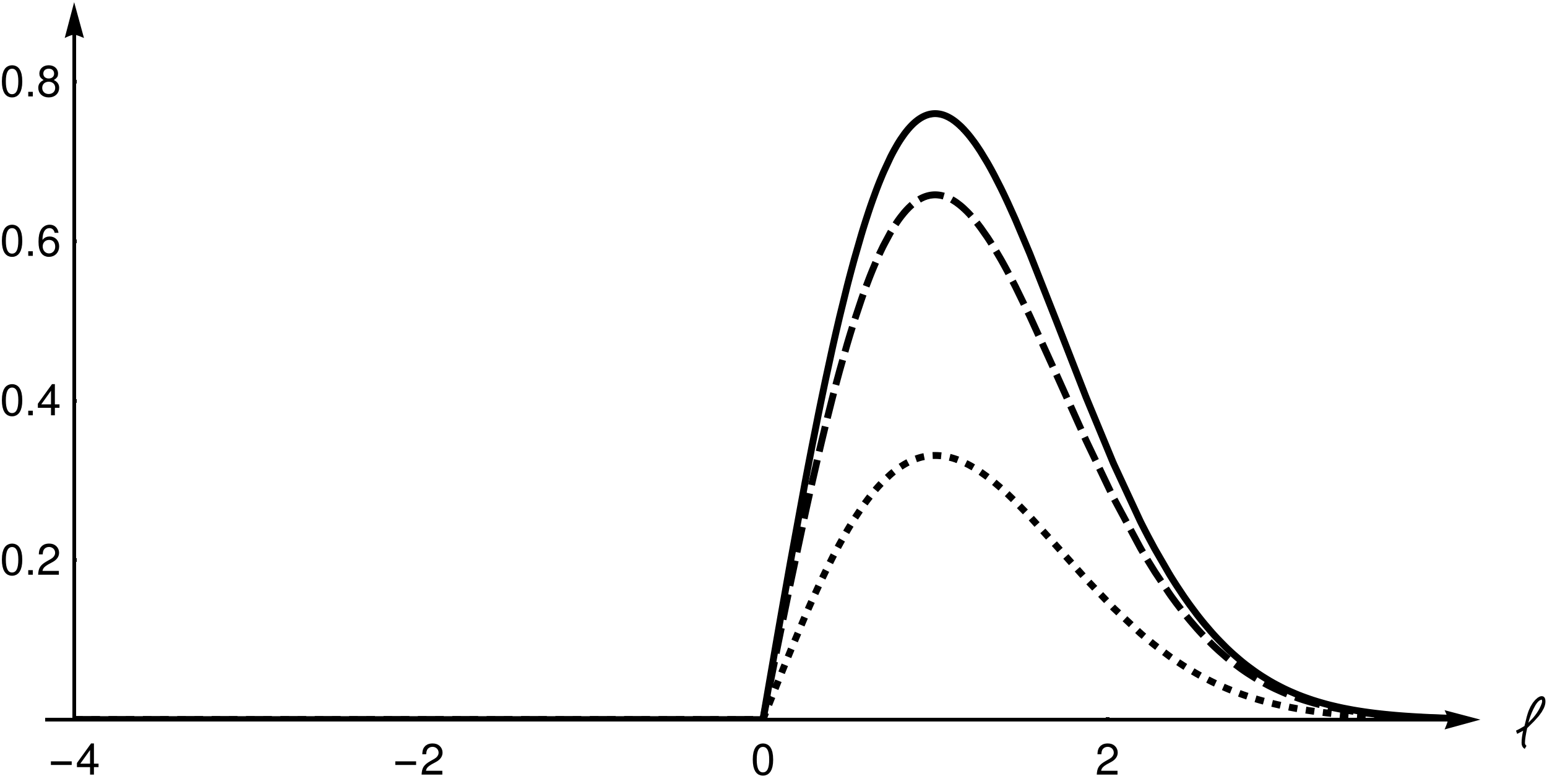}}\qquad
	\subfloat{\includegraphics[scale=0.5]{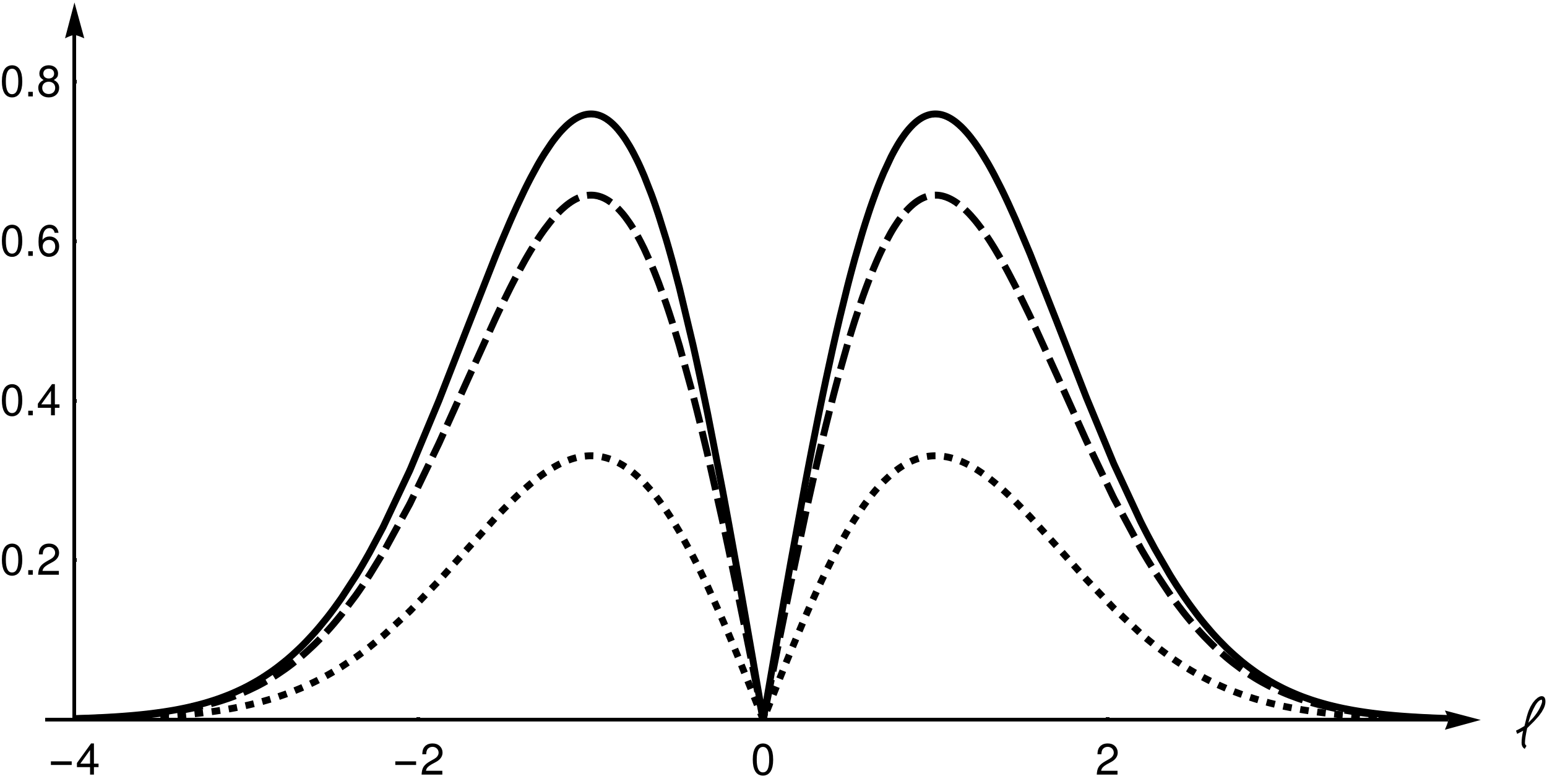}}
	\caption{The functions $c_{ES}(\ell)$ (left) and $c_{LS}(\ell)$ (right) with $\alpha=0$ for $\beta-\gamma=0$ (solid), $\beta-\gamma=0.5$ (dashed) and $\beta-\gamma=0.9$ (dotted) respectively.}\label{Fig_7+8}
\end{figure*}

\begin{figure*}[h!]
	\centering
	\subfloat[\label{Fig_9}]{\includegraphics[scale=0.5]{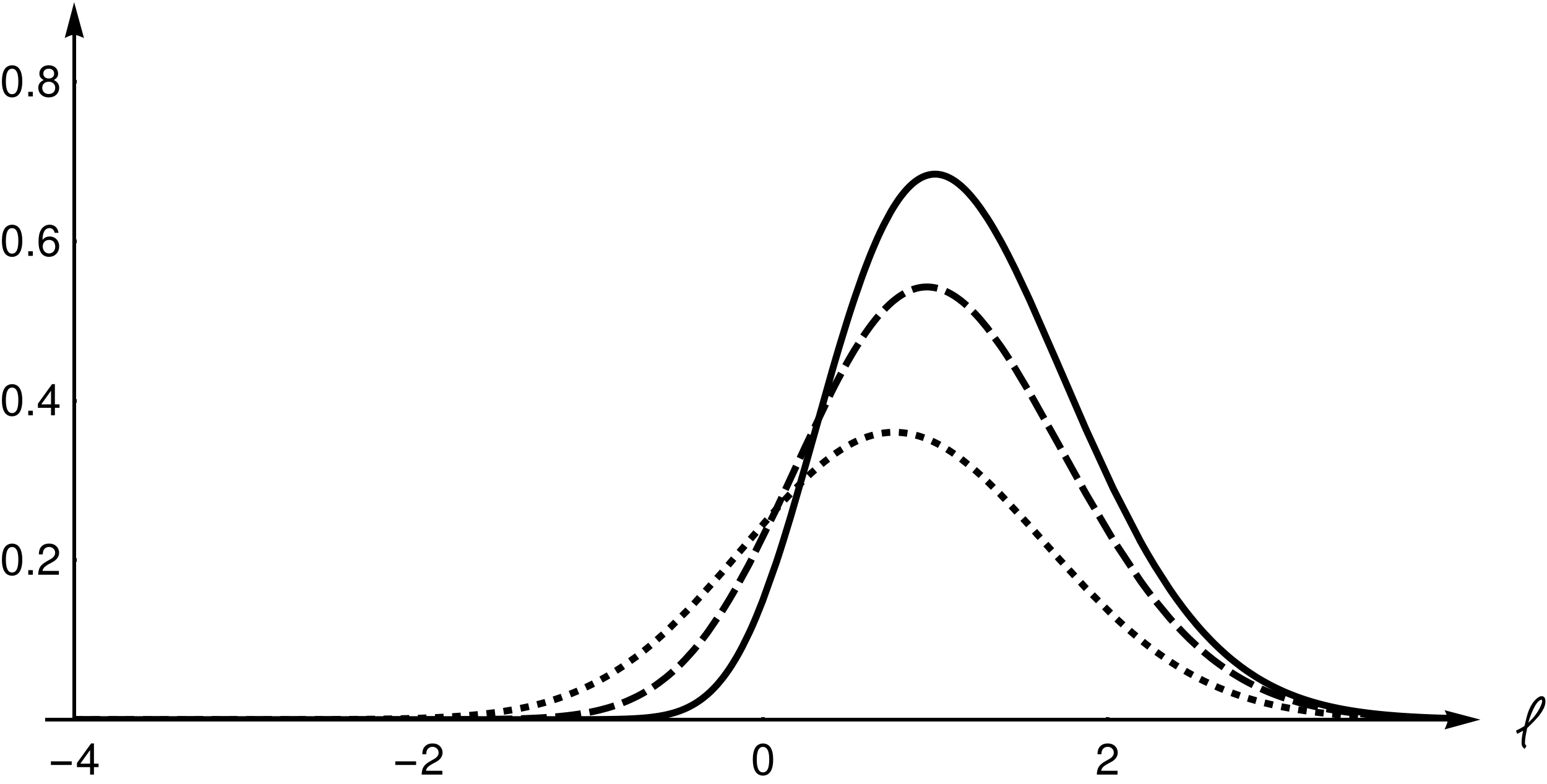}}\qquad
	\subfloat[\label{Fig_10}]{\includegraphics[scale=0.5]{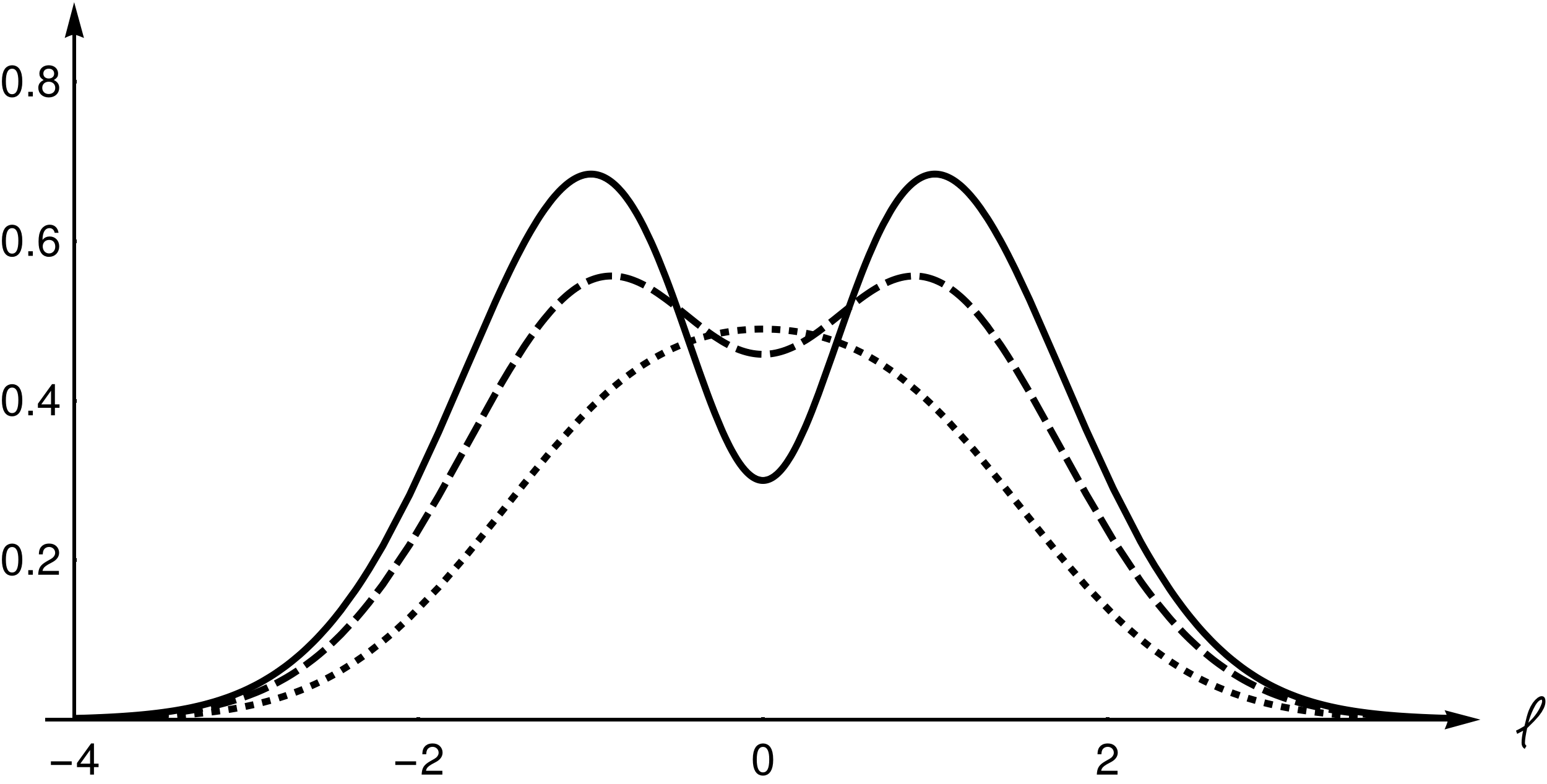}}
	\caption{The functions $c_{ES}(\ell)$ (left) and $c_{LS}(\ell)$ (right) with $\beta=\gamma$ for $\alpha=0.1$ (solid), $\alpha=0.3$ (dashed) and $\alpha=0.6$ (dotted) respectively.}\label{Fig_9+10}
\end{figure*}

\subsection{Summary of the rest of the paper}
The remainder of the paper is organised as follows. In Section~\ref{s:proof main results} we state the main (completely deterministic) topological lemma underlying our results (Lemma~\ref{l:main topological}) and combine this with Proposition~\ref{p:density existence} to prove Theorems~\ref{t:main level},~\ref{t:main excursion} and~\ref{t:integral equality} and their corollaries. In Section~\ref{s:critical points} we give the full definition of upper/lower connected saddle points and prove Proposition~\ref{p:density existence}. In Section~\ref{s:topological} we prove Lemma~\ref{l:main topological} via a series of steps, and also prove Proposition~\ref{p:cilleruello c_NS}, thereby completing the proof of all results in the paper.


\section{Proof of the main results}\label{s:proof main results}
The first step in proving Theorems~\ref{t:main level},~\ref{t:main excursion} and~\ref{t:integral equality} is to translate our assumptions on the Gaussian field $f$ into topological properties; these properties differ slightly depending on whether or not $f$ is periodic.
\begin{definition}\label{d:periodic}
	Let $g:\R^2\to\R$ be a function which is not constant in any direction; that is, there is no $v\in\R^2\backslash\{0\}$ such that for all $x\in\R^2$ and $a\in\R$, $g(x)=g(x+av)$. We say that $g$ is \textit{doubly periodic} if there exist two linearly independent vectors $y,z\in\R^2$ such that for all $x\in\R^2$, $g(x+y)=g(x)=g(x+z)$. We say that $g$ is \textit{singly periodic} if it is not doubly periodic and there exists $y\in\R^2\backslash\{0\}$ such that for all $x\in\R^2$, $g(x)=g(x+y)$. If neither of these conditions holds we say that $g$ is \textit{aperiodic}.
\end{definition}
We note that this definition applies to deterministic functions. We say that a random field $f$ is \textit{doubly periodic} if there exist linearly independent vectors $y,z\in\R^2$ such that with probability one, $f(x+y)=f(x)=f(x+z)$ for all $x\in\R^2$. We make an analogous definition for singly periodic fields, and we say that a random field is \textit{aperiodic} if it is neither doubly periodic nor singly periodic. The Gaussian field with spectral measure $\rho$ is doubly periodic if and only if there exists $A\in GL(2,\R)$ such that $\{Ax:x\in\text{spt}(\rho)\}\subset\mathbb{Z}^2$, and similarly is singly periodic if and only if it is not doubly periodic and there exists $A\in GL(2,\R)$ such that $\{Ax:x\in\text{spt}(\rho)\}\subset\mathbb{Z}\times\R$.

For our purposes, it is more natural to consider periodic functions as being defined on the torus or cylinder, and so we will specify the topological properties of Gaussian fields in terms of these domains. 

For any doubly periodic function $g:\R^2\to\R$ we can choose two linearly independent vectors $y,z\in\R^2$ satisfying the conditions in Definition~\ref{d:periodic} with minimum distance to the origin, and $g$ is then specified entirely by its values on the parallelogram
\begin{displaymath}
	P:=\left\{x\in\R^2:x=ty+sz,\;t,s\in[0,1)\right\} .
\end{displaymath}
We call this the \textit{associated parallelogram} for $g$ and call $y,z$ \textit{periodic vectors} for $g$ (they are not unique since we could also choose $-y,-z$). By identifying $P$ with the 2-dimensional torus~$\mathbb{T}$, we let $g_\mathbb{T}$ be $g$ defined on the torus. By the choice of $y,z$ above, when $g$ is a stationary random field there are no two distinct points in $P$ on which $g$ is almost surely equal.

If $g$ is singly periodic then we can choose $y\in\R^2\backslash\{0\}$ satisfying the conditions in Definition~\ref{d:periodic} with minimum distance to the origin; we call this a \textit{periodic vector} for $g$. By rotating the axes we may assume that $g$ is periodic in the direction of the $x$-axis (i.e.\ $y=(y_1,0)$) so that $g$ is specified entirely by its values on $[0,y_1]\times\R$. This strip can be identified with the infinite cylinder $\mathcal{C}:=S^1\times\R$ under the quotient relationship that $x\sim z$ if and only if $x-z=(my_1,0)$ for some $m\in\mathbb{Z}$. We will often work with the compact subset $\mathcal{C}(n):=[0,y_1]\times[-n,n]$ under the same quotient relationship. We let $g_\mathcal{C}$ and $g_{\mathcal{C}(n)}$ denote $g$ restricted to these surfaces under the previous identification.

We next introduce some elements of Morse theory, in particular the concept of a Morse function on a manifold; here we follow \cite{handron2002generalized}. We emphasise that, unless explicitly stated, a manifold $M$ may or may not have a boundary $\partial M$. We also emphasise that we will only ever work with $M$ being the plane $\mathbb{R}^2$, the closure of the disc $B(R)$, the torus $\mathbb{T}$, or the cylinders $\mathcal{C}$ and $\mathcal{C}(n)$, and so it is sufficient to bear these examples in mind.

\begin{definition}
	\label{d:morse}
	Let $M$ be an $n$-dimensional Riemannian manifold and $f\in C^2_\text{loc}(M)$. (If $\partial M\neq\emptyset$, we take this to mean that for any coordinate chart $\mathbf{x}$, the function $f\circ\mathbf{x}^{-1}$ can be extended to a $C^2$ function on an open subset of $\R^n$. In particular this means that $f|_{\partial M}$ is twice continuously differentiable.) We say that $f$ is \textit{Morse} if the following hold (with any condition depending on $\partial M$ holding implicitly if $M$ has no boundary):
	\begin{enumerate}
		\item The critical points of $f$ and $f|_{\partial M}$ are non-degenerate;
		\item None of the critical points of $f$ are contained in $\partial M$;
		\item If $\nabla f$ is tangent to $\partial M$ at $p$ and $u$ is a unit vector normal to $\partial M$ at $p$, then $\partial_{\nabla f(p)}\partial_u f(p)\neq 0$ where $\partial_{\nabla f(p)}$ and $\partial_u$ denote the directional derivatives in the directions $\nabla f(p)$ and $u$ respectively;
		\item The critical points of $f$ and $f|_{\partial M}$ all occur at distinct levels.
	\end{enumerate}
	If $M$ is a manifold without boundary this simplifies to the requirement that all critical points of $f$ are non-degenerate and occur at distinct levels. 
\end{definition}

We note that in the above definition $f|_{\partial M}$ is viewed as a function on the $(n-1)$-dimensional manifold $\partial M$ so that a critical point of $f|_{\partial M}$ is not necessarily a critical point of $f$. 

Using the definition of a Morse function, we next introduce a set of assumptions that a stationary random field must satisfy in order for our main results to hold. In the lemma that immediately follows, we will claim in particular that these assumptions are satisfied for Gaussian fields satisfying Conditions \ref{Conditions 1}, but we have chosen to isolate these assumptions to illustrate the limited role of Gaussianity in the proof of the main results.

We say that a point $t\in\partial M$ is a \emph{tangent point} of $f$ if $t$ is a critical point of the restricted function $f|_{\partial M}$ (the name comes from the fact that the level set $\{f=f(t)\}$ will be tangent to $\partial M$ at $t$). We will also use the term `tangent point' to describe a local extremum of the restriction of $f$ to a finite union of line segments. Let $N_\text{tang}(A)$ denote the number of tangent points of $f$ in $A$ (where $A$ is the boundary of a manifold or a finite union of line segments) and $N_\text{crit}(M)$ the number of critical points of $f$ in $M$. We will only use these terms when $A$ and $M$ are compact, so that by the first point in the definition of a Morse function, the number of critical and tangent points will be finite.

\begin{condition}\leavevmode\label{conditions main results}
	A stationary random field $f$ satisfies the following:
	\begin{enumerate}
		\item If $f$ is aperiodic, then for each $R>0$, $f|_{\overline{B(R)}}$ is almost surely Morse (taking a union over $R\in\mathbb{N}$, this implies also that $f$ is almost surely Morse on $\mathbb{R}^2$);
		\item If $f$ is doubly periodic, then $f_\mathbb{T}$ is almost surely Morse;
		\item If $f$ is singly periodic, then for each $n\in\mathbb{N}$, $f_{\mathcal{C}(n)}$ is almost surely Morse (taking a union over $\mathbb{N}$, this implies that $f_\mathcal{C}$ is almost surely Morse);
		\item If $f$ is periodic and $L\subset\R^2$ is a line segment of length $R$, then $\mathbb{E}(N_\text{tang}(L))=c_\theta R$ for a constant $c_\theta>0$ that depends only on the direction of $L$;
		\item $\mathbb{E}(N_\text{crit}(B(R)))=c R^2$ for a constant $c>0$;
		\item $\mathbb{E}(N_\text{tang}(\partial B(R)))=O(R)$ as $R\rightarrow\infty$.
	\end{enumerate}
\end{condition}

\begin{lemma}\label{Gaussian fields main theorem}
	Let $f$ satisfy Conditions \ref{Conditions 1}. Then $f$ satisfies Conditions~\ref{conditions main results}.
\end{lemma}

\begin{proof}
	
	All six conditions are established using variations of the Kac-Rice formula. This is an important tool in the study of Gaussian fields, and random functions more generally, which relates moments of the number of zeroes of a field to conditional moments of its gradient/derivative. See \cite{nicolaescu2014kac} for an introduction to the Kac-Rice formula or \cite[Chapter 11]{RFG} for a rigorous derivation.
	
	\textbf{(1)}-\textbf{(3)}.\ We first note that since $f\in C^2_\text{loc}(\R^2)$ almost surely by assumption, $f|_M\in C^2_\text{loc}(M)$ almost surely whenever $M$ is $\overline{B(R)}$, $\mathbb{T}$ or $\mathcal{C}(n)$. Corollary 11.3.5 of \cite{RFG} states that $f|_M$ has the first two properties of a Morse function if there exists a countable atlas $\{\mathbf{x}_i\}_{i\in I}$ such that for each $i\in I$, the covariance kernel of $f^{(i)}:=f|_M\circ\mathbf{x}_i^{-1}$ is $C^{4+}$ (on its domain of definition) and the Gaussian vectors $(\partial_xf^{(i)}(0),\partial_{xx}f^{(i)}(0),\partial_{xy}f^{(i)}(0))$ and $(\partial_yf^{(i)}(0),\partial_{yy}f^{(i)}(0),\partial_{xy}f^{(i)}(0))$ are non-degenerate (where $\partial_x$ denotes taking the derivative with respect to the first variable, $\partial_{xx}$ denotes taking the second derivative with respect to the first variable and we make analogous definitions for $\partial_y$, $\partial_{yy}$ and $\partial_{xy}$). For each choice of $M$ that we consider (i.e.\ $\overline{B(R)}$, $\mathbb{T}$ or $\mathcal{C}(n)$) it is clear that we can choose a finite atlas of charts $\mathbf{x}_i$ and that $f|_M\circ\mathbf{x}_i^{-1}$ will be a translation of $f$ (restricted to some open set) for each such chart. Therefore the covariance kernel of $f^{(i)}$ is a restriction of $\kappa$, which is $C^{4+}$ by assumption, and all that remains is to show that $(\partial_xf(t),\partial_{xx}f(t),\partial_{xy}f(t))$ and $(\partial_yf(t),\partial_{yy}f(t),\partial_{xy}f(t))$ are non-degenerate Gaussian vectors for any $t\in\R^2$. By assumption $\nabla f(t)$ and $\nabla^2 f(t)$ are non-degenerate Gaussian vectors for any $t\in\R^2$. It is a standard fact, for Gaussian fields with constant variance, that $\nabla f(t)$ and $\nabla^2 f(t)$ are independent for fixed $t$ (see e.g.\ \cite[Chapter 5]{RFG}) and this proves the necessary non-degeneracy. (This part of the result does not depend on whether $f$ is periodic.)
	
	We note that $f|_{\overline{B(R)}}$ has the third property of a Morse function provided that there is no $\theta\in[0,2\pi]$ such that
	\begin{align*}
		0=g_1(\theta)&:=\begin{pmatrix}
			\cos(\theta)\\
			\sin(\theta)
		\end{pmatrix}^T
		\nabla f(p)\\
		0=g_2(\theta)&:=
		\begin{pmatrix}
			\cos(\theta)\\
			\sin(\theta)
		\end{pmatrix}^T
		\nabla^2 f(p)
		\begin{pmatrix}
			-\sin(\theta)\\
			\cos(\theta)
		\end{pmatrix}
	\end{align*}
	where $p=(R\cos(\theta),R\sin(\theta))$ (see Figure~\ref{Fig_11}).
	
	\begin{figure*}[h!]
		\centering
		\input{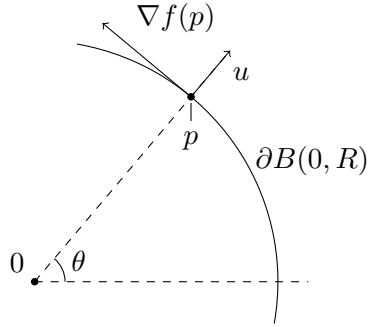}
		\caption{If $p=(R\cos(\theta),R\sin(\theta))$ and $\nabla f(p)$ is orthogonal to $u:=(\cos(\theta),\sin(\theta))$, we require that $\partial_{\nabla f(p)}\partial_u f(p)\neq 0$.}\label{Fig_11}
	\end{figure*}
	
	If we assumed that $f$ is $C^3$, then Bulinskaya's lemma would imply that no such $\theta$ exists almost surely. Since we assume only that $f$ is $C^{2+\nu}$, we use a different argument which is adapted from the proof of \cite[Lemma 11.2.11]{RFG}. Since $f\in C^{2+\nu}_\text{loc}(\R^2)$ almost surely, $g_2$ is almost surely $\nu$-H\"{o}lder continuous. Therefore for each $\delta>0$ we can find $C_\delta,D_\delta>0$ such that, with probability at least $1-\delta$, on $\partial B(R)$ the $\nu$-H\"{o}lder norm of $g_2$ is at most $C_\delta$ and each element of $\nabla^2 f$ is bounded in absolute value by $D_\delta$; we denote this event by $A_\delta$.
	
	Now let $I_{mj}$ be a collection of open intervals in $\R$ such that $[0,2\pi]\subset\cup_j I_{mj}$ for each $m$, $\sum_j\textup{Length}(I_{mj})\to 2\pi$ as $m\to\infty$, and $\textup{Length}(I_{mj})\to 0$ uniformly in $j$ as $m\to\infty$. Furthermore let $\theta_{mj}$ be the midpoint of each $I_{mj}$ and define
	\begin{displaymath}
		E_{mj}=\left\{\exists \theta\in I_{mj} : g_1(\theta)=g_2(\theta)=0\right\}.
	\end{displaymath}
	On the event $A_\delta\cap E_{mj}$ we note that
	\begin{displaymath}
		\lvert g_1(\theta_{mj})\rvert \leq c_1 D_\delta\textup{Length}(I_{mj})\quad\text{and}\quad\lvert g_2(\theta_{mj})\rvert\leq c_1 C_\delta\textup{Length}(I_{mj})^\nu
	\end{displaymath}
	for some universal constant $c_1>0$. Therefore
	\begin{equation}\label{e:Bulinskaya alternative}
		\begin{aligned}
			\mathbb{P}(&A_\delta\cap E_{mj})\\
			&\leq\int_{\lvert x\rvert
				\leq c_1 D_\delta\textup{Length}(I_{mj})}\mathbb{P}\left(\lvert g_2(\theta_{mj})\rvert\leq c_1 C_\delta\textup{Length}(I_{mj})^\nu\Big|g_1(\theta_{mj})=x\right)p_{g_1(\theta_{mj})}(x)\;dx\\
			&=\int_{\lvert x\rvert
				\leq c_1 D_\delta\textup{Length}(I_{mj})}\mathbb{P}\left(\lvert g_2(\theta_{mj})\rvert\leq c_1 C_\delta\textup{Length}(I_{mj})^\nu\right)p_{g_1(\theta_{mj})}(x)\;dx\\
		\end{aligned}
	\end{equation}
	where $p_{g_1(\theta_{mj})}$ is the probability density function of $g_1(\theta_{mj})$ and we have used the fact that $g_1(\theta_{mj})$ is independent of $g_2(\theta_{mj})$. Each $g_i(\theta)$ is a linear combination of elements of a non-degenerate Gaussian vector and so has non-zero variance, therefore on the compact region $[0,2\pi]$ these variances are bounded away from zero. Since these are Gaussian fields, their univariate densities are bounded above by a constant times the inverse of their standard deviations, so we see that the density of each $g_i(\theta_{mj})$ is bounded above uniformly in $m$ and $j$. We can therefore sum the expression on either side of \eqref{e:Bulinskaya alternative} over $j$. Since $\text{Length}(I_{mj})$ converges to zero uniformly in $j$, so does the integrand in each term of the sum, and so we conclude that
	\begin{displaymath}
		\sum_j\mathbb{P}(A_\delta\cap E_{mj})\to 0
	\end{displaymath}
	as $m\to\infty$. Since this is true for any $\delta>0$ and $\mathbb{P}(A_\delta)\geq 1-\delta$ we conclude that there is almost surely no $\theta\in[0,2\pi]$ such that $g_1(\theta)=g_2(\theta)=0$, and hence $f$ almost surely has the third property of a Morse function.

	The proof that $f_{\mathcal{C}(n)}$ satisfies the third property of a Morse function is near identical. The argument above can be repeated for the boundary $S^1\times\{-n,n\}$. Finally, $f_\mathbb{T}$ trivially satisfies the third property of a Morse function.
	
	Next we consider the fourth property of a Morse function. We first suppose $f$ is aperiodic and define $h_1:T_k\to\R^5$ where $T_k=\left\{(s,t)\in\overline{B(R)}^2\:\middle|\:\lvert s-t\rvert\geq 1/k\right\}$ by
	\begin{displaymath}
		h_1(s,t)=\begin{pmatrix}
			\nabla f(s)\\
			\nabla f(t)\\
			f(s)-f(t)
		\end{pmatrix} .
	\end{displaymath}
	Bulinskaya's lemma (Lemma~11.2.10 of \cite{RFG}) states that $\mathbb{P}(h_1^{-1}(0)=\emptyset)=1$ provided that $h_1\in C^1(T_k)$ almost surely and the univariate densities of $h_1$ are bounded uniformly in $T_k$. By assumption $h_1\in C^1_\text{loc}(\R^2)$ and so, since $h^1$ is Gaussian, we need only show that the determinant of the covariance matrix of $(\nabla f(s),\nabla f(t),f(s)-f(t))$ is bounded away from zero on $T_k$. Since $f$ is aperiodic, $f(s)-f(t)$ is a non-degenerate Gaussian variable for all $s\neq t$, and so by Conditions \ref{Conditions 1} (in particular using stationarity of $f$), the determinant of this covariance matrix is non-zero for all $s\neq t$. Since this determinant is continuous in $(s,t)$ it is bounded away from zero on the compact set $T_k$. Taking the countable union of $T_k$ for $k\in\mathbb{N}$ shows that almost surely there are no points $s,t\in B(R)$ with $s\neq t$ such that $\nabla f(s)=\nabla f(t)=0$ and $f(s)=f(t)$.
	
	Let $f_\partial(\theta)$ be a parametrisation of $f|_{\partial B(R)}$. To show that $f$ has no two tangent points at the same level and no tangent points at the level of any critical point, we apply identical arguments to the following two functions:
	\begin{displaymath}
		h_2(\theta,\omega)=\begin{pmatrix}
			f^\prime_\partial(\theta)\\
			f^\prime_\partial(\omega)\\
			f_\partial(\theta)-f_\partial(\omega)
		\end{pmatrix}
		\quad \text{and} \quad
		h_3(t,\theta)=
		\begin{pmatrix}
			\nabla f(t)\\
			f^\prime_\partial(\theta)\\
			f(t)-f_\partial(\theta)
		\end{pmatrix} .
	\end{displaymath}
	This completes the proof that $f|_{\overline{B(R)}}$ is a Morse function when $f$ is aperiodic. When $f$ is doubly periodic, we can repeat these arguments with
	\begin{displaymath}
		T_k=\left\{(s,t)\in P_k^2:\lvert s-t\rvert\geq 1/k\right\}
	\end{displaymath}
	where $P_k=\{ay+bz:a,b\in[0,1-1/k]\}$ and $y,z$ are the periodic vectors of $f$, to show that $f_\mathbb{T}$ is Morse. Finally in the singly periodic case we consider $C_k=[0,y_1-1/k]\times[-n,n]$ and $T_k=\left\{(s,t)\in C_k^2:\lvert s-t\rvert\geq 1/k\right\}$ where $y=(y_1,0)$ is the periodic vector of $f$. In both of these cases, the choice of domain ensures that $f(t)-f(s)$ is non-degenerate, so that Conditions \ref{Conditions 1} can be used. This completes the proof of the first three points of the lemma.
	
	\textbf{(4)}.\ Let $L$ be a line segment of length $R$ and $u$ a unit length vector parallel to $L$. Applying a Kac-Rice formula (for the particular form, see Corollary 11.2.2 of \cite{RFG}) to $f|_L$ and using the stationarity of $f$ along with the independence of $\nabla f(t)$ and $\nabla^2 f(t)$, we see that
	\begin{align*}
		\mathbb{E}(N_\text{tang}(L))&=\int_L\mathbb{E}(\lvert\partial_{uu} f(t)\rvert\:|\:\partial_u f(t)=0)p_{\partial_u f(t)}(0) \, dt\\
		&=R\; p_{\partial_u f(0)}(0)\mathbb{E}(\lvert\partial_{uu} f(0)\rvert) =: c_\theta R
	\end{align*}
	where $p_{\partial_u f(0)}$ is the probability density of $\partial_u f(0)$, and $c_\theta<\infty$ since $\partial_{uu} f(0)$ can be expressed as a linear combination of components of the non-degenerate Gaussian vector $\nabla^2 f(0)$ and so has finite moments of all orders.
	
	\textbf{(5)}.\ Applying the same Kac-Rice formula to $f$, and using the same arguments as above, shows that
	\begin{align*}
		\mathbb{E}(N_\text{crit}(B(R)))&=\int_{B(R)}\mathbb{E}\left(\left\lvert\det\nabla^2 f(t)\right\rvert\:\middle|\:\nabla f(t)=0\right)p_{\nabla f(t)}(0) \, dt\\
		&=\pi R^2p_{\nabla f(0)}(0)\mathbb{E}\left(\left\lvert\det\nabla^2 f(0)\right\rvert\right) =: c R^2.
	\end{align*}
	
	\textbf{(6)}.\ We now explicitly parametrise $f|_{\partial B(R)}$ as
	$g(\theta):=f(R\cos(\theta),R\sin(\theta))$ and note that
	\begin{equation}
		\begin{aligned}\label{e:tangent point bound}
			g^\prime(\theta)=&-R\sin(\theta)\partial_xf(p)+R\cos(\theta)\partial_yf(p)\\
			g^{\prime\prime}(\theta)=&-R\cos(\theta)\partial_xf(p)-R\sin(\theta)\partial_yf(p)\\
			&+R^2\sin^2(\theta)\partial_{xx}f(p)-2R^2\sin(\theta)\cos(\theta)\partial_{xy}f(p)+R^2\cos^2(\theta)\partial_{yy}f(p)
		\end{aligned}
	\end{equation}
	where $p=(R\cos(\theta),R\sin(\theta))$. Since $\nabla f(p)$ and $\nabla^2 f(p)$ are independent, a simple calculation shows that $(g^\prime(\theta),g^{\prime\prime}(\theta))$ has non-degenerate distribution, so we can apply a Kac-Rice formula (specifically, Corollary 11.2.2 of \cite{RFG}) to show that
	\begin{equation}\label{e:kac rice tangent}
		\begin{aligned}
			\mathbb{E}(N_\text{tang}(\partial B(R)))&=
			\mathbb{E}\left(\#\{\theta\in[0,2\pi]\:\big|\:g^\prime(\theta)=0\}\right)\\
			&=\int_0^{2\pi}\mathbb{E}\left(\lvert g^{\prime\prime}(\theta)\rvert\:\big|\:g^\prime(\theta)=0\right)p_{g^\prime(\theta)}(0)\, d\theta .
		\end{aligned}
	\end{equation}
	Since $f$ is stationary
	\begin{displaymath}
		\text{Var}\left(g^\prime(\theta)\right)=R^2\mathbb{E}\left(\left(-\sin\left(\theta\right)\partial_xf(0)+\cos\left(\theta\right)\partial_yf(0)\right)^2\right).
	\end{displaymath}
	The expectation on the right side of this equation is non-zero for each $\theta$, since $\nabla f(0)$ is non-degenerate, and by compactness is bounded away from zero uniformly in $\theta$. Therefore, by considering the density of a Gaussian random variable, there exists $c_0>0$ such that $p_{g^\prime(\theta)}(0)\leq c_0/R$ for all $\theta\in[0,2\pi]$ and $R>0$.

	Using this observation, along with the fact that $\nabla f(0)$ is independent of $\nabla^2 f(0)$, and substituting \eqref{e:tangent point bound} into \eqref{e:kac rice tangent}, we see that there exists a constant $c_1>0$ independent of $\theta$ and $R$ such that
	\begin{align*}
		\mathbb{E}(N_\text{tang}(\partial B(R)))&\leq c_1 R\int_0^{2\pi }\mathbb{E}\left(\lvert \partial_{xx}f(0)\rvert+\lvert \partial_{yy}f(0)\rvert+\lvert \partial_{xy}f(0)\rvert\right)d\theta\\
		&\quad+c_1\int_0^{2\pi}\mathbb{E}\left(\lvert \partial_xf(0)\rvert+\lvert \partial_yf(0)\rvert\:\big|\:\sin(\theta)\partial_xf(p)=\cos(\theta)\partial_yf(p)\right)d\theta\\
		&=O(R) . 
	\end{align*}
\end{proof}

We now introduce the deterministic relationship between level sets and critical points that is the foundation of our results. We denote the number of critical points of $f$ in $B(R)$ of type $h$ with level in $[\ell_1,\ell_2]$ by $N_{h,R}[\ell_1,\ell_2]$, where $h=m^+,m^-,s^+,s^-$ denotes local maxima, local minima, upper connected saddles and lower connected saddles respectively (for an aperiodic function that is Morse on $\R^2$, Definition~\ref{d:lower connected aperiodic} defines the upper/lower connected saddle points; the definitions in the general case will be given in Section~\ref{s:critical points}). Let us also generalise our earlier notation: for a deterministic, planar function $f$, we denote the number of components of $\{f=\ell\}$ and $\{f\geq\ell\}$ in $B(R)$ by $N_{LS,R}(\ell)$ and $N_{ES,R}(\ell)$ respectively.

In the case of a singly periodic function $f:\R^2\to\R$ we also need to define internal and external rectangular tilings of $B(x,R)$ which have a finite number of horizontal and vertical line segments as boundary. Let $(y_1, 0)$ be the periodic vector of $f$, and let $S(n_1,n_2)=[n_1y_1,(n_1+1)y_1]\times[n_2,n_2+1]$ where $n_1,n_2\in\mathbb{Z}$. We define $B_\text{int}(x,R)$ to be the union over $n_1,n_2\in\mathbb{Z}$ of all $S(n_1,n_2)$ contained in $B(x,R)$. Similarly we define $B_\text{ext}(x,R)$ to be the union over $n_1,n_2\in\mathbb{Z}$ of all $S(n_1,n_2)$ which intersect $B(x,R)$. 

\begin{lemma}\label{l:main topological}
	Let $f:\R^2\rightarrow\R$ be a deterministic function. Suppose first that $f$ is aperiodic and assume that $f$ and $f|_{\overline{B(R)}}$ are Morse. Then there exists $c>0$, independent of $f$, such that for each $\ell\in\R$,
	\begin{equation}
		\label{e:top1}
		N_{LS,R}(\ell)= N_{m^+,R}[\ell,\infty)-N_{s^-,R}[\ell,\infty) +N_{s^+,R}(\ell,\infty)-N_{m^-,R}(\ell,\infty)+\eta_{\ell,R}
	\end{equation}
	and
	\begin{equation}
		\label{e:top2}
		N_{ES,R}(\ell)= N_{m^+,R}[\ell,\infty)-N_{s^-,R}[\ell,\infty)+\gamma_{\ell,R}
	\end{equation}
	where
	\begin{equation}
		\label{e:top3}
		\max\{\lvert\eta_{\ell,R}\rvert,\lvert\gamma_{\ell,R}\rvert\}\leq c N_\text{tang}(\partial B(R)).
	\end{equation}
	Suppose instead that $f:\R^2\rightarrow\R$ is doubly periodic and assume that $f_\mathbb{T}$ is Morse and $f|_{\partial P}$ has a finite number of local extrema. Then there exists a constant $c_f > 0$ depending only on $P$, such that, for each $R > 0$ and $\ell\in\R$,
	\eqref{e:top1} and \eqref{e:top2} hold with \eqref{e:top3} replaced by
	\begin{equation*}
		\max\{\lvert\eta_{\ell,R}\rvert,\lvert\gamma_{\ell,R}\rvert\}\leq c_f\cdot\left(N_\text{crit}(P)\cdot R+N_\text{tang}(\partial B(R))\right) .
	\end{equation*}
	Finally, suppose that $f:\R^2\rightarrow\R$ is singly periodic with periodic vector $(y,0)$ and assume that $f_{\mathcal{C}(n)}$ is Morse for each $n\in\mathbb{Z}$ and that $N_\text{tang}(\{0\}\times[-2R,2R])<\infty$. Then \eqref{e:top1} and \eqref{e:top2} still hold with \eqref{e:top3} replaced by
	\begin{align*}
		\max\{\lvert\eta_{\ell,R}\rvert,\lvert\gamma_{\ell,R}\rvert\}\leq c_f\Bigg(
		N_\text{tang}&(\partial B_\text{int}(R))+N_\text{tang}(\partial B_\text{ext}(R))+N_\text{crit}(B(R+r_f)\backslash B(R-r_f))\Bigg)
	\end{align*}
	where $c_f$ and $r_f$ are constants that depend only on the periodic vector $y=(y_1,0)$ of $f$ (so in particular, they are independent of $\ell$ and $R$).
	
\end{lemma}

\noindent The proof of Lemma~\ref{l:main topological} is contained in Section~\ref{s:topological}. The error term in the above estimate can be intuitively understood as the result of boundary effects from working with the domain $B(R)$ in the definition of $N_{LS, R}$ and $N_{ES, R}$, and although it appears quite complicated (especially in the case of singly periodic functions), when applied to the Gaussian field $f$ it is bounded in expectation by $O(R)$ as a result of Lemma~\ref{Gaussian fields main theorem}.

We are now ready to state the core technical results of the paper, which are versions of Theorems~\ref{t:main level},~\ref{t:main excursion} and~\ref{t:integral equality} that hold for arbitrary random fields (i.e.\ not necessarily Gaussian) using only the properties contained in Conditions~\ref{conditions main results}.

\begin{proposition}\label{NS integral equality}
	Let $f$ be a stationary random field satisfying Conditions~\ref{conditions main results}. For each $\ell\in\R$
	\begin{displaymath}
		\mathbb{E}\left(N_{LS,R}(\ell)\right)=c_{LS}(\ell)\cdot\pi R^2+O(R)
	\end{displaymath}
	as $R\to\infty$, where
	\begin{displaymath}
		c_{LS}(\ell):=\frac{1}{\pi}\mathbb{E}\left(N_{m^+,1}[\ell,\infty)-N_{s^-,1}[\ell,\infty)+N_{s^+,1}(\ell,\infty)-N_{m^-,1}(\ell,\infty)\right).
	\end{displaymath}
	The constant implied by the $O(\cdot)$ notation may depend on the distribution of $f$ but is independent of $\ell$. If in addition $f$ is ergodic, then
	\begin{displaymath}
		\frac{1}{\pi R^2}N_{LS,R}(\ell)\xrightarrow{L^1,a.s.}c_{LS}(\ell).
	\end{displaymath}
\end{proposition}

\begin{proposition}\label{ES integral equality}
	Let $f$ be a stationary random field satisfying Conditions~\ref{conditions main results}. For each $\ell\in\R$
	\begin{displaymath}
		\mathbb{E}\left(N_{ES,R}(\ell)\right)=c_{ES}(\ell)\cdot\pi R^2+O(R)
	\end{displaymath}
	as $R\to\infty$, where
	\begin{displaymath}
		c_{ES}(\ell):=\frac{1}{\pi}\mathbb{E}(N_{m^+,1}[\ell,\infty)-N_{s^-,1}[\ell,\infty)).
	\end{displaymath}
	The constant implied by the $O(\cdot)$ notation may depend on the distribution of $f$ but is independent of $\ell$. If in addition $f$ is ergodic, then
	\begin{displaymath}
		\frac{1}{\pi R^2}N_{ES,R}(\ell)\xrightarrow{L^1,a.s.}c_{ES}(\ell).
	\end{displaymath}
\end{proposition}
\begin{remark}
	Since $f$ is stationary, we could equivalently define $c_{LS}(\ell)$ in Proposition~\ref{NS integral equality} as
	\begin{displaymath}
		c_{LS}(\ell)=\frac{1}{\pi R^2}\mathbb{E}\left(N_{m^+,R}[\ell,\infty)-N_{s^-,R}[\ell,\infty)+N_{s^+,R}(\ell,\infty)-N_{m^-,R}(\ell,\infty)\right)
	\end{displaymath}
	for any $R>0$ or make an analogous definition for $c_{ES}(\ell)$ in Proposition~\ref{ES integral equality}.
\end{remark}

Together with Lemma~\ref{Gaussian fields main theorem} and Proposition~\ref{p:density existence}, these propositions imply Theorems~\ref{t:main level} and~\ref{t:main excursion}, and also the first two parts of Theorems ~\ref{t:integral equality}.

\begin{proof}[Proposition~\ref{NS integral equality}]
	As $f$ is stationary, the expected number of critical points of $f$ of a particular type in a domain is proportional to the area of the domain. Since the field satisfies Conditions~\ref{conditions main results} we can apply Lemma~\ref{l:main topological}, and taking expectations yields
	\begin{displaymath}
		\mathbb{E}(N_{LS,R}(\ell))=\pi R^2c_{LS}(\ell)+O(R) . 
	\end{displaymath}
	Sending $R\to\infty$ proves the first two statements of the proposition.
	
	The remainder of the proof follows the general roadmap of the original derivation of the existence of the Nazarov-Sodin constant in \cite{SodinNazarov2015asymptotic}. Suppose that $f$ is ergodic and let $N_{h,R}^{(u)}$ denote the number of critical points of type $h$ in $B(u,R)$ with level in $[\ell,\infty)$ or $(\ell,\infty)$ for $h=m^+,s^+$ and $h=m^-,s^-$ respectively. The `sandwich estimate' of \cite{SodinNazarov2015asymptotic} (Lemma~1) can be slightly altered to show that for any $r\in(0,R)$
	\begin{displaymath}
		\frac{1}{\pi R^2}\int_{B(R-r)}\frac{N_{h,r}^{(u)}}{\pi r^2} \, du\leq\frac{N_{h,R}^{(0)}}{\pi R^2}\leq\frac{1}{\pi R^2}\int_{B(R+r)}\frac{N_{h,r}^{(u)}}{\pi r^2} \, du .
	\end{displaymath}
	Applying Wiener's ergodic theorem (see \cite[Section 6.1]{SodinNazarov2015asymptotic}) both of these integrals converge almost surely and in $L^1$ to the limit $\mathbb{E}\left(N_{h,1}^{(0)}\right)/\pi$. So in particular, $N_{h,R}^{(0)}/(\pi R^2)$ has the same limit, and
	\begin{displaymath}
		\frac{1}{\pi R^2}\left(N_{m^+,R}^{(0)}+N_{s^+,R}^{(0)}-N_{s^-,R}^{(0)}-N_{m^-,R}^{(0)}\right)\xrightarrow{L^1}c_{LS}(\ell). 
	\end{displaymath}
	Applying Lemma~\ref{l:main topological} with the bound on $\mathbb{E}(\lvert\eta_{\ell,R}\rvert)$ implied by Conditions~\ref{conditions main results} shows that
	\begin{align*}
		\frac{1}{\pi R^2}\mathbb{E}\left(\left\lvert N_{LS,R}(\ell)-N_{m^+,R}^{(0)}-N_{s^+,R}^{(0)}+N_{s^-,R}^{(0)}+N_{m^-,R}^{(0)}\right\rvert\right)=o(1)
	\end{align*}
	as $R\rightarrow\infty$. Combining these results completes the proof of $L^1$ convergence.
	
	We now extend our notation by defining $N_{LS,r}^{(u)}$ to be the number of components of $\{f=\ell\}$ in $B(u,r)$. The original `sandwich estimate' of \cite{SodinNazarov2015asymptotic} states that
	\begin{displaymath}
		\frac{1}{\pi R^2}\int_{B(R-r)}\frac{N_{LS,r}^{(u)}}{\pi r^2}\, du\leq\frac{N_{LS,R}^{(0)}}{\pi R^2}\leq\frac{1}{\pi R^2}\int_{B(R+r)}\frac{\tilde{N}_{LS,r}^{(u)}}{\pi r^2} \, du
	\end{displaymath}
	where $\tilde{N}_{LS,r}^{(u)}$ is the number of components of $\{f=\ell\}$ which intersect $\overline{B(u,r)}$. Since 
	\begin{displaymath}
		\tilde{N}_{LS,r}^{(u)}=N_{LS,r}^{(u)}+O\left(\#\{x\in\partial B(u,r)\:|\:f(x)=\ell\}\right)
	\end{displaymath}
	we can rearrange this estimate as
	\begin{equation}\label{e:main theorem 2}
		\begin{aligned}
			\left\lvert\frac{N_{LS,R}^{(0)}}{\pi R^2}-\frac{1}{\pi R^2}\int_{B(R)}\frac{N_{LS,r}^{(u)}}{\pi r^2}\, du\right\rvert\leq &\frac{1}{\pi R^2}\int_{B(R+r)\backslash B(R-r)}\frac{N_{LS,r}^{(u)}}{\pi r^2} \, du\\
			&+\frac{c_0}{\pi R^2}\int_{B(R+r)}\frac{\#\{x\in\partial B(u,r)\:|\:f(x)=\ell\}}{\pi r^2}\, du
		\end{aligned}
	\end{equation}
	for some universal constant $c_0>0$. Applying Lemma~\ref{l:main topological} inside the integral term on the left hand side of \eqref{e:main theorem 2} shows that
	\begin{equation}
		\begin{aligned}\label{e:main theorem}
			&\left\lvert\frac{N_{LS,R}^{(0)}}{\pi R^2}-\frac{1}{\pi R^2}\int_{B(R)}\frac{N_{m^+,r}^{(u)}}{\pi r^2}+\frac{N_{s^+,r}^{(u)}}{\pi r^2}-\frac{N_{s^-,r}^{(u)}}{\pi r^2}-\frac{N_{m^-,r}^{(u)}}{\pi r^2}du\right\rvert\\
			&\quad\qquad\qquad\qquad\qquad\qquad\leq\frac{1}{\pi R^2}\int_{B(R+r)\backslash B(R-r)}\frac{N_{LS,r}^{(u)}}{\pi r^2} \, du\\
			&\quad\qquad\qquad\qquad\qquad\qquad\quad+\frac{c_0}{\pi R^2}\int_{B(R+r)}\frac{\#\{x\in\partial B(u,r)\:|\:f(x)=\ell\}}{\pi r^2} \, du\\
			&\quad\qquad\qquad\qquad\qquad\qquad\quad+\frac{c_f}{\pi R^2}\int_{B(R)}\frac{N_\text{crit}(P)\cdot r+N_\text{tang}(B(u,r))}{\pi r^2} \, du\\
			&\quad\qquad\qquad\qquad\qquad\qquad\quad+\frac{c_f}{\pi R^2}\int_{B(R)}\frac{N_\text{tang}(B_\text{int}(u,r))+N_\text{tang}(B_\text{ext}(u,r))}{\pi r^2}\, du\\
			&\quad\qquad\qquad\qquad\qquad\qquad\quad+\frac{c_f}{\pi R^2}\int_{B(R)}\frac{N_\text{crit}(B(u,r+r_f)\backslash B(u,r-r_f))}{\pi r^2} \, du .
		\end{aligned}
	\end{equation}
	(The upper bound here is a result of adding the upper bounds on the error terms in Lemma~\ref{l:main topological} in the cases that $f$ is aperiodic, doubly periodic and singly periodic respectively, so that we can deal with all three cases at once.) From Wiener's ergodic theorem, the integral term within the absolute value signs will converge almost surely to $c_{LS}(\ell)$ and the first integral term on the right hand side will converge almost surely to zero. Applying the same argument to the remaining integral terms shows that they will each converge to a constant, $\Phi_1(r)$, $\Phi_2(r)$, $\Phi_3(r)$ and $\Phi_4(r)$ respectively, where
	\begin{align*}
		\Phi_1(r)&=\frac{c_0}{\pi r^2}\mathbb{E}\left(\#\{x\in\partial B(r)\:|\:f(x)=\ell\}\right) , \\
		\Phi_2(r)&=\frac{c_f\mathbb{E}(N_\text{crit}(P))}{\pi r}+\frac{c_f}{\pi r^2}\mathbb{E}\left(N_\text{tang}(B(r))\right) , \\
		\Phi_3(r)&=\frac{c_f}{\pi r^2}\left(\mathbb{E}\left(N_\text{tang}(B_\text{int}(0,r))\right)+\mathbb{E}\left(N_\text{tang}(B_\text{ext}(0,r))\right)\right) , \\
		\Phi_4(r)&=\frac{c_f}{\pi r^2}\mathbb{E}\left(N_\text{crit}(B(u,r+r_f)\backslash B(u,r-r_f))\right) .
	\end{align*}
	We now fix $r>0$ and take the limit superior of \eqref{e:main theorem} as $R\to\infty$ to show that
	\begin{equation}\label{e:main theorem 3}
		\limsup_{R\rightarrow\infty}\left\lvert\frac{N_{LS,R}^{(0)}}{\pi R^2}-c_{LS}(\ell)\right\rvert\leq \Phi_1(r)+\Phi_2(r)+\Phi_3(r)+\Phi_4(r)
	\end{equation}
	almost surely. Since the number of boundary points of $f$ at level $\ell$ is deterministically bounded above by a constant times the number of tangent points, we see that $\Phi_1(r)\leq c_1\Phi_2(r)$ for some $c_1>0$. Conditions~\ref{conditions main results} imply that $\Phi_2(r)+\Phi_3(r)+\Phi_4(r)=o(1)$ as $r\rightarrow\infty$. Therefore we may take a countable sequence $r_n\rightarrow\infty$, such that \eqref{e:main theorem 3} holds almost surely for each $r_n$ and the right hand side of \eqref{e:main theorem 3} becomes arbitrarily small, to show that $N_{LS,R}^{(0)}/(\pi R^2)$ converges almost surely to $c_{LS}(\ell)$.
\end{proof}

\begin{proof}[Proof (Proposition~\ref{ES integral equality})]
	This follows the proof of Proposition~\ref{NS integral equality} almost exactly. Taking expectations of the second part of Lemma~\ref{l:main topological}, using the stationarity of $f$ and the bound on $\mathbb{E}(N_\text{tang}(\partial B(R)))$ implied by Conditions~\ref{conditions main results} proves the first two statements of the theorem.
	
	The proof of the remainder of the theorem follows identically by defining $N_{ES,r}^{(u)}$ as the number of components of $\{f\geq\ell\}$ contained in $B(u,r)$, $\tilde{N}_{ES,r}^{(u)}$ as the number of components of $\{f\geq\ell\}$ intersecting $\overline{B(u,r)}$ and noting that the difference in these two terms is also bounded above by the number of tangent points of $\partial B(u,r)$.
\end{proof}

To complete the proof of Theorem~\ref{t:integral equality}, it remains only to show the joint continuity of $c_{LS}$ and $c_{ES}$ with respect to the level and spectral measure.

\begin{proof}[Proof (Theorem~\ref{t:integral equality} - Joint continuity)]
	By Prokhorov's theorem the weak-$*$ topology on $\mathcal{P}_c$ is metrisable and so we can use the sequential definition of continuity for $c_{LS}$ and $c_{ES}$. Let $(\rho_n,\ell_n)$ be a sequence in $\mathcal{P}_c\times\R$ converging to $(\rho,\ell)$. By the triangle inequality
	\begin{equation}\label{Joint continuity 1}
		\lvert c_{LS}(\rho_n,\ell_n)-c_{LS}(\rho,\ell)\rvert\leq\lvert c_{LS}(\rho_n,\ell_n)-c_{LS}(\rho_n,\ell)\rvert+\lvert c_{LS}(\rho_n,\ell)-c_{LS}(\rho,\ell)\rvert . 
	\end{equation}
	Theorem 1.3 of \cite{kurlberg2017variation} states that the second term on the right hand side converges to $0$ with $n$ in the special case $\ell=0$. However the proof of this theorem can be repeated verbatim replacing the field $f_\rho$ with $f_\rho-\ell$ to show that $\lvert c_{LS}(\rho_n,\ell)-c_{LS}(\rho,\ell)\rvert\to 0$.
	
	Now assume that $\ell_n\leq\ell$. By the first part of Theorem~\ref{t:integral equality} we see that
	\begin{equation}\label{Joint continuity 2}
		\lvert c_{LS}(\rho_n,\ell_n)-c_{LS}(\rho_n,\ell)\rvert\leq (4/\pi)\mathbb{E}_{\rho_n}\left(N_\text{crit}^{(n)}[\ell_n,\ell]\right)
	\end{equation}
	where $N_\text{crit}^{(n)}[\ell_n,\ell]$ denotes the number of critical points of $f_n$ (the Gaussian field with spectral measure $\rho_n$) in the circle of radius $1$ with level in $[\ell_n,\ell]$. By the Kac-Rice theorem (Corollary 11.2.2 of \cite{RFG})
	\begin{equation}
		\begin{aligned}\label{Joint continuity 3}
			\mathbb{E}_{\rho_n}\left(N_\text{crit}^{(n)}[\ell_n,\ell]\right)&=\int_{B(1)}\mathbb{E}\left(\left\lvert\det\nabla^2 f_n(t)\right\rvert\id_{f_n(t)\in[\ell_n,\ell]}\:\middle|\:\nabla f_n(t)=0\right)p_{\nabla f_n(t)}(0) \, dt\\
			&=\pi\,\mathbb{E}\left(\left\lvert\det\nabla^2 f_n(0)\right\rvert\id_{f_n(0)\in[\ell_n,\ell]}\right)p_{\nabla f_n(0)}(0)\\
			&\leq\mathbb{E}\left(\left\lvert\det\nabla^2 f_n(0)\right\rvert^2\right)^\frac{1}{2}\mathbb{P}
			\left(f_n(0)\in[\ell_n,\ell]\right)^\frac{1}{2}\frac{1}{2}\frac{1}{\sqrt{\det\text{Var}(\nabla f_n(0))}}
		\end{aligned}
	\end{equation}
	where we have used the independence of a field and its gradient at a point along with the Cauchy-Schwarz inequality. It is well known that the covariance structure of a stationary Gaussian field along with its derivatives at a point can be expressed in terms of the spectral measure (see \cite[Chapter 5]{RFG}). For example,
	\begin{displaymath}
		\text{Var}(\partial_x f_n(0))=\int_{\R^2}\lambda_1^2\: d\rho_n(\lambda)
	\end{displaymath}
	with similar expressions for other derivatives and covariances. Since all spectral measures we are considering are supported on $\overline{B(1)}$, the definition of weak-$*$ convergence implies that the covariance structure associated with each $\rho_n$ at the origin converges to that of $\rho$. Specifically, if $f$ denotes the field with spectral measure $\rho$, then
	\begin{displaymath}
		\text{Var}(f_n(0),\nabla f_n(0),\nabla^2 f_n(0))\rightarrow\text{Var}(f(0),\nabla f(0),\nabla^2 f(0)) .
	\end{displaymath}
	Applying this to \eqref{Joint continuity 3}, we see that $\det\text{Var}(\nabla f_n(0))$ is bounded away from $0$ in $n$. Similarly $\mathbb{E}\left(\lvert\det\nabla^2 f_n(0)\rvert^2\right)$ is uniformly bounded above in $n$. Finally we note that $\text{Var}(f_n(0))$ is uniformly bounded away from $0$, so that choosing $\ell_n$ sufficiently close to $\ell$ ensures that $\mathbb{P}\left(f_n(0)\in[\ell_n,\ell]\right)$ is arbitrarily small. An identical argument works for $\ell_n\geq\ell$, and so combining these observations with\ \eqref{Joint continuity 1}--\eqref{Joint continuity 3} shows that
	\begin{displaymath}
		\lvert c_{LS}(\rho_n,\ell_n)-c_{LS}(\rho,\ell)\rvert\xrightarrow{n\to\infty}0
	\end{displaymath}
	as required. The proof of continuity of $c_{ES}(\rho,\ell)$ is almost identical. Although Theorem 1.3 of \cite{kurlberg2017variation} is stated only for level sets, the proof of this result can be adapted to apply to excursion sets with no changes of any substance.
\end{proof}

\begin{proof}[Proof (Corollary~\ref{c:monotonicity})]
	Since $f$ is isotropic we can use the explicitly derived critical point densities for local maxima, local minima and saddle points from \cite{cheng2015expected}. We recall that these are parametrised in terms of $\lambda\in (0, \sqrt{2}]$ and $\eta^2\in [0, \infty)$ as defined by \eqref{e:isotropic parameters}. If $0<\lambda<\sqrt{2}$ then
	\begin{align*}
		p_{m^+}(x)=p_{m^-}(-x)&=\frac{1}{\pi\eta^2}\Bigg(\lambda^2(x^2-1)\phi(x)\Phi\left(\frac{\lambda x}{\sqrt{2-\lambda^2}}\right)+\frac{\lambda x\sqrt{2-\lambda^2}}{2\pi}e^{-\frac{x^2}{2-\lambda^2}}\\
		&\qquad\qquad\qquad\qquad+\frac{\sqrt{2}}{\sqrt{\pi(3-\lambda^2)}}e^{-\frac{3x^2}{2(3-\lambda^2)}}\Phi\left(\frac{\lambda x}{\sqrt{(3-\lambda^2)(2-\lambda^2)}}\right)\Bigg) \\
		p_s(x)&=\frac{1}{\pi\eta^2}\frac{\sqrt{2}}{\sqrt{\pi(3-\lambda^2)}}e^{-\frac{3x^2}{2(3-\lambda^2)}}
	\end{align*}
	where $\phi$ and $\Phi$ denote the standard normal probability density function and cumulative density function respectively. If $\lambda=\sqrt{2}$ then
	\begin{align*}
		p_{m^+}(x)=p_{m^-}(-x)&=\frac{\sqrt{2}}{\pi^{3/2}\eta^2}\left((x^2-1)e^{-\frac{x^2}{2}}+e^{-\frac{3x^2}{2}}\right)\id_{x\geq 0}\\
		p_s(x)&=\frac{\sqrt{2}}{\pi^{3/2}\eta^2}e^{-\frac{3x^2}{2}} . 
	\end{align*}
	By Theorem~\ref{t:integral equality},
	\begin{align*}
		c_{LS}(\ell+\epsilon)-c_{LS}(\ell)\leq \int_\ell^{\ell+\epsilon} p_{m^-}(x)+p_s(x)-p_{m^+}(x) \, dx .
	\end{align*}
	We denote the integrand above by $I(x)$. In the case $\lambda\in(0,\sqrt{2})$, evaluating this expression using the standard Gaussian inequality $1-\Phi(x)\leq(1/x)\phi(x)$ shows that
	\begin{displaymath}
		I(x)\leq \frac{1}{\pi\eta^2}\left(\frac{\sqrt{2-\lambda^2}}{\pi}\frac{1}{x}\frac{2-\lambda^2}{\lambda}e^{-\frac{x^2}{2-\lambda^2}}-\frac{\lambda^2}{\sqrt{2\pi}}(x^2-1)e^{-\frac{x^2}{2}}\right)
	\end{displaymath}
	which is negative for $x>\sqrt{2}/\lambda$. For $\lambda=\sqrt{2}$ and $x>0$
	\begin{displaymath}
		I(x)\leq-\frac{1}{4\pi\sqrt{2\pi}}(x^2-1)e^{-\frac{x^2}{2}}
	\end{displaymath}
	which is negative for $x>1=\sqrt{2}/\lambda$. Similarly we have
	\begin{displaymath}
		c_{ES}(\ell+\epsilon)-c_{ES}(\ell)\leq\int_\ell^{\ell+\epsilon} p_s(x)-p_{m^+}(x) \, dx ,
	\end{displaymath}
	which is less than our upper bound for $c_{LS}(\ell+\epsilon)-c_{LS}(\ell)$, and the first result follows.

	When $\lambda=\sqrt{2}$ and $\ell\leq 0$ we see that $p_{m^+}(\ell)=0$ and so by Theorem~\ref{t:integral equality}
	\begin{displaymath}
		c_{ES}(\ell+\epsilon)-c_{ES}(\ell)=\int_\ell^{\ell+\epsilon}p_{s^-}(x) \, dx\geq 0 .
	\end{displaymath}
	Therefore $c_{ES}$ is weakly increasing in this case.
	
\end{proof}


\section{Critical point densities}\label{s:critical points}
In this section we prove Proposition~\ref{p:density existence}. We begin by giving the definition for lower and upper connected saddle points in full generality. Let $M$ be a manifold and let $A\subset M$. We say that $A$ is \textit{simple} if it is compact, connected and every loop in $A$ (i.e.\ every continuous map $h:S^1\to A$) is $M$-contractible. We will be solely interested in the case that $A$ is a component of an excursion set of $f:M\to\R$. For such $A$, and in the case that $M$ is simply connected (e.g.\ for $M =\mathbb{R}^2$ or $B( R)$), the condition of being simple is just the same as being bounded.

\begin{definition}\label{d:lower connected general}
	Let $M$ be a $2$-dimensional Riemannian manifold without boundary and let $f:M\to\R$ be a Morse function with a saddle point $x\in M$ such that $f(x)=\ell$. For $c\leq\ell$, let $A_c$ denote the component of $\{f\geq c\}$ containing $x$ and for $c>\ell$ let $A_c=A_\ell\cap\{f\geq c\}$. We say that $x$ is \textit{lower connected} if either of the following conditions hold for $\epsilon>0$ sufficiently small:
	\begin{enumerate}
		\item $A_{\ell-\epsilon}$ is simple and $A_{\ell+\epsilon}$ consists of two simple components; or
		\item $A_{\ell-\epsilon}$ is not simple but $A_{\ell+\epsilon}$ has a simple component.
	\end{enumerate}
	We say that $x$ is \textit{upper connected} if it is a lower connected saddle for $-f$.
\end{definition}

\begin{figure*}[h!]
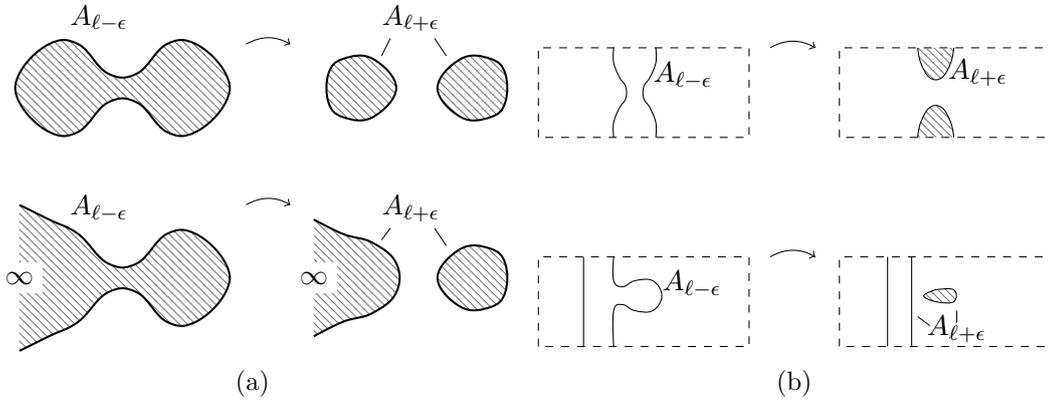

	\centering
	\subfloat[\label{Fig_12}]{\input{Fig_12.tikz}}\qquad
	\subfloat[\label{Fig_13}]{\input{Fig_13.tikz}}
	\caption{Subfigures~\ref{Fig_12} and~\ref{Fig_13} give two examples of the excursion sets (shaded) at a level below and above a lower connected saddle point in $\R^2$ and $\mathbb{T}$ respectively (the dashed rectangles in~\ref{Fig_13} are identified with the torus by the standard quotient relation).}\label{Fig_12+13}
\end{figure*}

When $f$ is an aperiodic Gaussian field, by Lemma~\ref{Gaussian fields main theorem} it will be Morse on $\R^2$ so we may take this as our Riemannian manifold and use the definition of lower/upper connected saddle points directly (we show in Section~\ref{s:topological} that this definition coincides with that given in Section~\ref{s:introduction}). If $f$ is doubly periodic, recall that it is completely specified by its values on the associated parallelogram $P$ which we identify with the torus. We then say that a saddle point of $f$ is lower connected if it corresponds to a lower connected saddle of $f_\mathbb{T}$ by the definition above (see Figure~\ref{Fig_14}). Similarly if $f$ is singly periodic we say that a saddle point of $f$ is lower connected if it corresponds to a lower connected saddle point of $f_\mathcal{C}$. Upper connected saddle points are defined analogously.

\begin{figure*}[h!]
	\centering
	\input{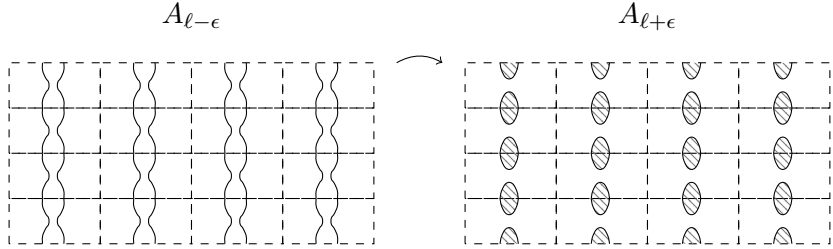}
	\caption{Passing through the lower connected saddle points of a doubly periodic function increases the number of compact excursion sets.}\label{Fig_14}
\end{figure*}

The following lemma shows that our above definitions partition the set of saddle points in all cases of interest.

\begin{lemma}\label{l:no infinite four arm}
	Let $f$ be a stationary random field satisfying Conditions~\ref{conditions main results}. Then with probability one all saddle points of $f$ are either upper connected or lower connected but not both.
\end{lemma}
This result is proven in Section~\ref{s:topological}. We are now ready to prove Proposition~\ref{p:density existence}.

\begin{proof}[Proof (Proposition~\ref{p:density existence})]
	Let $f$ satisfy Conditions \ref{Conditions 1}. We fix a compact $\Omega\subset\R^2$ such that $\partial\Omega$ has finite Hausdorff-$1$ measure and consider $N_h(\ell)$, the number of critical points of $f|_\Omega$ of type $h$ with value greater than $\ell$, where $h=m^-,s,m^+$ correspond to local minima, saddle points and local maxima respectively. If $(f(0),\nabla^2 f(0))$ has a non-degenerate distribution, then by a Kac-Rice formula (specifically, Corollary 11.2.2 of \cite{RFG}, which requires that $\partial\Omega$ has finite Hausdorff-1 measure)
	\begin{equation}
		\begin{aligned}\label{Density existence}
			\mathbb{E}(N_h(\ell))&=\int_\Omega\mathbb{E}\left(\left\lvert\det\nabla^2 f(t)\right\rvert\id_{\text{Index}\nabla^2 f(t)=i,f(t)>\ell}\:\middle|\:\nabla f(t)=0\right)p_{\nabla f(t)}(0) \, dt\\
			&=\text{Area}(\Omega)p_{\nabla f(0)}(0)\mathbb{E}\left(\left\lvert\det\nabla^2 f(0)\right\rvert\id_{\text{Index}\nabla^2 f(0)=i,f(0)>\ell}\right)\\
			&=\text{Area}(\Omega)p_{\nabla f(0)}(0)\int_\ell^\infty\mathbb{E}\left(\left\lvert\det\nabla^2 f(0)\right\rvert\id_{\text{Index}\nabla^2 f(0)=i}\:\middle|\:f(0)=u\right)\phi(u) \, du
		\end{aligned}
	\end{equation}
	where $i$ is the index corresponding to $h$ and $\phi$ denotes the standard Gaussian probability density. The third equality follows from the definition of conditioning on a Gaussian variable. If $(f(0),\nabla^2 f(0))$ has a degenerate distribution then $f(0)$ can be expressed as a linear combination of the elements of $\nabla^2 f(0)$ almost surely. Substituting in this expression for $f(0)$ allows us to apply Kac-Rice and the arguments above to derive \eqref{Density existence} in this case too. Then, by Gaussian regression, $\mathbb{E}(\lvert\det\nabla^2 f(0)\rvert\id_{\text{Index}\nabla^2 f(0)=i}\:|\:f(0)=u)$ is a continuous function of $u$ (see, e.g., Proposition~1.2 of \cite{azais2009level}). This proves the existence and continuity of the densities $p_{m^+},p_{m^-}$ and $p_s$, and the fact that $p_{m^+}(x) = p_{m^-}(-x)$ follows from the symmetry of $f$.
	
	Since $\mathbb{E}(N_{s}[0,\ell,])=\int_0^\ell p_s(x)dx$, we know that $\mathbb{E}(N_{s}[0,\ell])$ is absolutely continuous in $\ell$. It is also clear that for any $\ell_1<\ell_2$, $\mathbb{E}(N_{s^+}[\ell_1,\ell_2])\leq\mathbb{E}(N_{s}[\ell_1,\ell_2])$ so $\mathbb{E}(N_{s^+}[0,\ell])$ is absolutely continuous in $\ell$. Therefore there exists a function $p_{s^+}:\R\rightarrow[0,\infty)$ such that
	\begin{displaymath}
		\mathbb{E}(N_{s^+}[0,\ell])=\int_0^\ell p_{s^+}(x) \, dx.
	\end{displaymath}
	Since $\mathbb{E}(N_s[0,\infty))$ is finite, the monotone convergence theorem shows that $p_{s^+}\in L^1(\R)$. By symmetry of the Gaussian distribution, and the definition of lower connected saddles, this also shows the existence of $p_{s^-}(x)=p_{s^+}(-x)$. The fact that $p_{s^-}+p_{s^+}=p_s$ follows from Lemma~\ref{l:no infinite four arm}.
\end{proof}


\section{Topological lemmas}\label{s:topological}
\noindent In this section we prove the deterministic Lemma~\ref{l:main topological} using topological arguments. We also establish Lemma~\ref{l:no infinite four arm} and Proposition~\ref{p:cilleruello c_NS} using similar methods, thereby completing the proof of all results in the paper.

To prove Lemmas~\ref{l:main topological} and~\ref{l:no infinite four arm} we require several results from Morse theory which we now introduce. We note that several aspects of this theory require us to work with compact manifolds, and so in the aperiodic and singly periodic cases we work with $B( R)$ and $\mathcal{C}(n)$ rather than directly with $\mathbb{R}^2$ and $\mathcal{C}$; in the doubly periodic case, by contrast, we work directly with the torus $\mathbb{T}$. We also emphasise that we only ever work with the five examples of manifolds just mentioned, so it is sufficient to have them in mind.

Recall that, for a Morse function $f$ defined on a manifold $M$, the tangent points of $f$ are defined as the critical points of~$f|_{\partial M}$.

\begin{theorem}[Theorem 7 of \cite{handron2002generalized}]\label{Morse 1}
	Let $M$ be a compact $n$-dimensional Riemannian manifold and let $f:M\to\R$ be a Morse function. If $f$ has no critical or tangent points with value in $[a,b]$ then $\{f\geq b\}$ is homotopy equivalent to $\{f\geq a\}$.
\end{theorem}

We define a $k$-cell to be a copy of the closed unit disc in $\R^k$ and temporarily denote this by $B_k$. If $Y$ is a topological space then we define the following operation to be `attaching a $k$-cell to $Y$'. First we find a continuous function $g:\partial B_k\to Y$, then we take the disjoint union $Y\sqcup B_k$ and identify each point in $\partial B_k$ with its image under $g$. By attaching a $0$-cell, we simply mean taking the disjoint union of $Y$ and a single point.
\begin{theorem}[Theorem 8 of \cite{handron2002generalized}]\label{Morse 2}
	Let $M$ be a compact $2$-dimensional Riemannian manifold and let $f:M\to\R$ be a Morse function. If $t$ is a critical point of $f$ of index $k$ with $f(t)=\ell$, then for $\epsilon>0$ sufficiently small, $\{f\geq\ell-\epsilon\}$ is homotopy equivalent to $\{f\geq\ell+\epsilon\}$ with a $(2-k)$-cell attached. If $t$ is a tangent point and $\ell,\epsilon$ are defined in the same way, then $\{f\geq\ell-\epsilon\}$ is homotopy equivalent to either $\{f\geq\ell+\epsilon\}$ or $\{f\geq\ell+\epsilon\}$ with a $k$-cell attached for some $k\in\{0,1,2\}$.
\end{theorem}

These theorems are proven using methods very similar to those of the standard proofs for manifolds without boundary which can be found in almost any text on Morse theory. 

We now work towards a proof of Lemma~\ref{l:main topological}. First we will need another definition which, in the case of aperiodic or singly periodic fields, identifies a subset of the upper/lower connected saddle points which have unfavourable topological properties (for our purposes). Recall that if $M$ is a manifold then $A\subset M$ is said to be simple if $A$ is compact, connected and every loop in $A$ is $M$-contractible. In the case that $M = \mathbb{R}^2$ or $B( R)$ and $A$ is an excursion set component, this is just the condition that $A$ is bounded.

\begin{definition}
	Let $M_\infty$ be a Riemannian $2$-manifold without boundary, with $f_\infty\in C^2_\text{loc}(M_\infty)$ a Morse function on $M_\infty$, and let $M$ be a compact submanifold of $M_\infty$ with boundary, with $f:=f_\infty|_M$ a Morse function on $M$. Let $x_0$ be a saddle point of $f$ at level $\ell$ and let $A_c$ be defined as in Definition~\ref{d:lower connected general} for $f_\infty$. If $x_0$ is lower connected, then we say that it is \textit{four-arm in} $M$ if, for $\epsilon>0$ sufficiently small, all of the simple components of $A_{\ell+\epsilon}$ intersect $\partial M$. Similarly, we say that an upper connected saddle is \textit{four-arm in} $M$ if it satisfies the previous condition for $-f$ (and $-\ell$).
	
	We say that $x_0$ is an \textit{infinite-four-arm saddle point} if $A_{\ell+\epsilon}$ has two unbounded components for all $\epsilon>0$ sufficiently small.
\end{definition}

\begin{remark}
	In the proof of Lemma~\ref{l:no infinite four arm}, we will show that when $M_\infty=\R^2$ (the most important case for our analysis), the level set at the level of an infinite-four-arm saddle takes the form shown in Figure~\ref{Fig_16}. This corresponds to the way an `infinite-four-arm event' is typically defined in the percolation literature: as two disjoint paths joining a point to infinity which are separated by two `dual' paths joining the same point to infinity. For other choices of $M_\infty$ (such as $M_\infty=\mathcal{C}$) the level set at the level of an infinite-four-arm saddle may look quite different, so that this terminology is less intuitive.
\end{remark}

Let us explain the importance of four-arm saddles. If $f$ is aperiodic, then its lower connected saddle points are defined so as to correspond to an increase in the number of excursion set components in $\R^2$. However if such a saddle point is four-arm in $\overline{B(R)}$, then this increase cannot be observed from inside $\overline{B(R)}$ since the excursion sets that are created when passing through the saddle intersect $\partial B(R)$ (see Figure~\ref{Fig_15+16}). The case when $f$ is singly periodic is similar (in the case that $f$ is doubly periodic we do not have to worry about such saddles). Infinite-four-arm saddle points will be relevant when we prove Lemma~\ref{l:no infinite four arm}. Fortunately we can control the number of four-arm saddles which occur, in terms of the boundary behaviour of $f$. 

\begin{figure*}[h!]
	\centering
	\subfloat[An upper connected saddle that is four-arm in $\overline{B(R)}$.]{\label{Fig_15}\input{Fig_15.tikz}}\qquad
	\subfloat[An infinite-four-arm saddle point in $\R^2$.]{\label{Fig_16}\input{Fig_16.tikz}}
	\caption{Different types of four-arm saddle points.}\label{Fig_15+16}
\end{figure*}

\begin{lemma}\label{l:four arm}
	Let $M_\infty$ be a Riemannian $2$-manifold without boundary and $f_\infty\in C^2_\text{loc}(M_\infty)$. Let $M$ be a compact submanifold of $M_\infty$ with boundary and $f:=f_\infty|_M$ be Morse. Let $N_\text{4-arm}(M)$ be the number of saddle points of $f_\infty$  contained in $M$ which are four-arm in $M$ or infinite-four-arm. Then
	\begin{displaymath}
		N_\text{4-arm}(M)\leq 3 N_\text{tang}(\partial M) .
	\end{displaymath}
\end{lemma}

Intuitively this bound follows because as we raise the level past a saddle which is four-arm in $\overline{B(R)}$ or an infinite-four-arm saddle we separate two components of $\{f\geq\ell\}\cap\partial B(R)$ (i.e.\ the boundary components are no longer in the same excursion set component). The total number of such separations which may occur is bounded above by the number of boundary excursion components at different levels, which in turn is bounded by the number of tangent points. This is formalised in the proof below.

\begin{proof}
	We describe an algorithm which will create an injective mapping from the lower connected saddle points which are four-arm in $M$ to the tangent points of $f$ in $\partial M$. Let $x_1,\dots,x_m$ denote all such saddle points arranged in order of increasing level, i.e.\ $f(x_i)<f(x_{i+1})$. We fix $\epsilon>0$ such that there are no tangent points with value in $[f(x_i)-\epsilon,f(x_i)+\epsilon]$ for any $i$ and these intervals are disjoint for different $i$. For each $i=1,\dots,m$, we define $S_1^{(i)},\dots,S_{m_i}^{(i)}$ to be the simple components of $\{f\geq f(x_i)+\epsilon\}$ which intersect $\partial M$.
	
	We proceed by induction. Let $A_-$ denote the component of $\{f\geq f(x_1)-\epsilon\}$ containing~$x_1$. Since $x$ is lower connected and four-arm in $M$, $A_-$ will contain some $S_j^{(1)}$ (it may contain two such elements: see Figure~\ref{Fig_17}), and so we make a preliminary association between $x_1$ and $S_j^{(1)}$.
	
	\begin{figure*}[h!]
		\centering
		\input{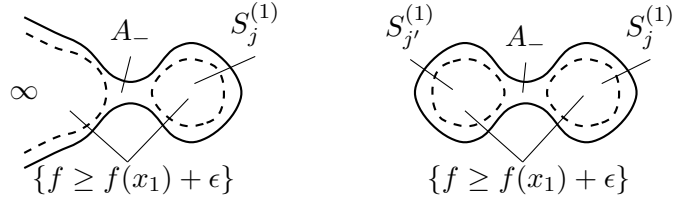}
		\caption{Since $x_1$ is a lower connected saddle point, $A_-$ may contain either one (left) or two (right) simple component of $\{f\geq f(x_1)+\epsilon\}$.}\label{Fig_17}
	\end{figure*}
	
	We now assume inductively that for each $x_1,\dots,x_{n}$ we have either a preliminary association with some $S_j^{(n)}$ (which differs across points) or a final association to some tangent point with value less than $f(x_{n})$.
	
	We consider $x_{n+1}$ and let $A_-$ denote the component of $\{f\geq f(x_{n+1})-\epsilon\}$ containing $x_{n+1}$. First suppose that $A_-$ is contained in some $S_j^{(n)}$ which has a preliminary association with some $x_i$ (where $i\leq n$). Then $A_-$ must be simple, as it is a subset of $S_j^{(n)}$, and so, by the definition of a lower connected four-arm saddle point, this means that $A_-$ must contain two different components $S_p^{(n+1)}$ and $S_q^{(n+1)}$. We then make a preliminary association between $x_{n+1}$ and $S_p^{(n+1)}$ and make a new preliminary association between $x_i$ and $S_q^{(n+1)}$ (so we remove the old association between $x_i$ and $S_j^{(n)}$). If $A_-$ is contained in some $S_j^{(n)}$ which does not have a preliminary association or if $A_-$ is not contained in any $S_j^{(n)}$ then we choose some $S_p^{(n+1)}$ contained in $A_-$ (which must exist since $x_{n+1}$ is four-arm in $M$) and make a preliminary association between this and $x_{n+1}$.
	
	Now suppose that $x_i$, where $i\leq n$, has a preliminary association with $S_j^{(n)}$. If the maximum level of any point in $S_j^{(n)}$ is less than $f(x_{n+1})$ then we make a final association between $x_i$ and the highest such point in $\partial M\cap S_j^{(n)}$, which is by definition a tangent point. We also remove the preliminary association between $x_i$ and $S_j^{(n)}$. Otherwise $S_j^{(n)}$ must be a superset of $S_k^{(n+1)}$ for some $k$, and we then make a new preliminary association between $x_i$ and $S_k^{(n+1)}$. We perform this process for all preliminary associations which remain after the step described in the previous paragraph, which completes the inductive step. This algorithm will cease after the $m$-th step, at which stage if a point $x_i$ has a preliminary association with some $S_j^{(n+1)}$ we make a final association between $x_i$ and the highest tangent point of $S_j^{(n+1)}$.
	
	The end result of this algorithm is an injective mapping, defined by the final associations, from the set of lower connected saddles which are four-arm in $M$ to the tangent points of $f$. An identical argument applied to $-f$ creates such a map for the the upper connected saddles which are four-arm in $M$. A very similar algorithm can be applied to the infinite-four-arm saddles of $f$, except that the inductive step is slightly simpler. Combining these results proves the lemma.
\end{proof}

Our next result is a preliminary version of Lemma~\ref{l:main topological} for Morse functions on manifolds without boundary. This setting allows us to prove the desired relationship between the number of excursion set components and critical points in the simpler case that all critical points have different levels. 

\begin{lemma}\label{l:excursion sets manifolds}
	Let $M_\infty$ be a Riemannian $2$-manifold without boundary and $f\in C^2_\text{loc}(M_\infty)$ be a Morse function. Let $M$ be a compact submanifold of $M_\infty$ and $f:=f_\infty|_M$ be Morse. For $\ell\in\R$, let $N_\text{simple}(\ell)$ be the number of simple components of $\{f\geq\ell\}$ which do not intersect the boundary of $M$ and let $N_{m^+}(\ell)$ and $N_{s^-}(\ell)$, be the number of local maxima and lower connected saddle points respectively of $f$ with level in $[\ell,\infty]$. Then
	\begin{displaymath}
		N_\text{simple}(\ell)=N_{m^+}(\ell)-N_{s^-}(\ell)+\zeta
	\end{displaymath}
	where $\lvert\zeta\rvert\leq 5N_\text{tangent}$. So in particular, if $M$ has no boundary then $\zeta=0$.
\end{lemma}

The proof involves applying the Morse theorems to each critical/tangent point to derive the change in the number of simple components at the corresponding levels and then summing these changes.

\begin{proof}
	First suppose that $f$ has no critical or tangent points with level in $[a,b]$. By Theorem~\ref{Morse 1},~$\{f\geq b\}$ is a deformation retract of $\{f\geq a\}$ under some map $h:\{f\geq a\}\times[0,1]\to\{f\geq a\}$. In particular, this is also true for each component of $\{f\geq b\}$ and the component of $\{f\geq a\}$ in which it is contained. Since $h$ is a homotopy, the number of simple components is the same in each set. If $A$ is a component of $\{f\geq a\}$ which intersects $\partial M$, then we claim that $A\cap\{f\geq b\}$ also intersects $\partial M$. To see why, suppose $A$ intersects $\partial M$ but $A\cap\{f\geq b\}$ does not. Then by considering the infimum $\ell^*$ of $\ell\in[a,b]$ such that $A\cap\{f\geq\ell\}$ does not intersect $\partial M$ and taking a sequence of points in $A\cap\{f\geq\ell^*-1/n\}\cap\partial M$ we see that $A\cap\{f\geq\ell^*\}$ contains a tangent point at level $\ell^*\in[a,b]$, which is a contradiction. Combining all these observations, we see that $N_\text{simple}$ is constant on intervals which contain no critical or tangent points.
	
	Now let $x$ be a critical or tangent point of $f$ at level $c$ and take $\epsilon>0$ small enough to apply Lemma~\ref{Morse 2}. We consider in turn the different types of critical or tangent point and calculate $N_\text{simple}(c-\epsilon)-N_\text{simple}(c+\epsilon)$. Let $A_{c-\epsilon}$ denote the component of $\{f\geq c-\epsilon\}$ containing $x$ and $A_{c+\epsilon}=A_{c-\epsilon}\cap\{f\geq c+\epsilon\}$. We note that to determine $N_\text{simple}(c-\epsilon)-N_\text{simple}(c+\epsilon)$ it is enough to consider $A_{c-\epsilon}$ and $A_{c+\epsilon}$, since by the arguments in the previous paragraph, the number of simple components of $\{f\geq\ell\}$ in $M\backslash A_{c-\epsilon}$ which do not intersect $\partial M$ will be constant as $\ell$ varies in $[c-\epsilon,c+\epsilon]$.
	
	If $x$ is a local maximum, then by Lemma~\ref{Morse 2}, $A_{c-\epsilon}$ is homotopy equivalent to $A_{c+\epsilon}$ with a $0$-cell attached. So in particular $\{f\geq c-\epsilon\}$ has one more component than $\{f\geq c+\epsilon\}$, and this extra component is $M$-contractible. Since the extra component contains a local maximum, (which cannot be in $\partial M$) for $\epsilon>0$ sufficiently small, this component is disjoint from $\partial M$ and so $N_\text{simple}(c-\epsilon)-N_\text{simple}(c+\epsilon)=1$.
	
	If $x$ is a local minimum, then $A_{c-\epsilon}$ is homotopy equivalent to $A_{c+\epsilon}$ with a $2$-cell attached. Attaching a $2$-cell does not change the number of components of $\{f\geq c+\epsilon\}$, and does not affect whether the component it is attached to is simple or not. Once again since the local minimum is not in $\partial M$, restricting $\epsilon$ sufficiently small ensures that the attached $2$-cell does not intersect $\partial M$ and so $N_{\text{simple}}(c-\epsilon)-N_{\text{simple}}(c+\epsilon)=0$.
	
	Next we suppose that $x$ is a saddle point, so that by Lemma~\ref{Morse 2}, $A_{c-\epsilon}$ is homotopy equivalent to $A_{c+\epsilon}$ with a $1$-cell attached. So in particular, $A_{c+\epsilon}$ consists of either one or two components. First we note that if $A_{c-\epsilon}$ is simple and does not intersect $\partial M$, then any components of $A_{c+\epsilon}$ have both these properties, therefore $N_{\text{simple}}(c-\epsilon)-N_{\text{simple}}(c+\epsilon)\leq 0$. If $A_{c-\epsilon}$ is simple but intersects $\partial M$, then by repeating the sequential argument in the first paragraph of this proof we see that $A_{c+\epsilon}$ must intersect $\partial M$ and so cannot consist of two simple components disjoint from the boundary. Furthermore, if $A_{c-\epsilon}$ is not simple, then by the above homotopy, $A_{c+\epsilon}$ cannot consist of two simple components. Combining these two observations show that $N_{\text{simple}}(c-\epsilon)-N_{\text{simple}}(c+\epsilon)\geq -1$. So the number of simple components disjoint from $\partial M$ is either constant or increases by one on passing through the saddle point. If $x$ is not a lower connected saddle, then by definition $N_{\text{simple}}(c-\epsilon)-N_{\text{simple}}(c+\epsilon)=0$, if $x$ is a lower connected saddle which is not four-arm in $M$, then $N_{\text{simple}}(c-\epsilon)-N_{\text{simple}}(c+\epsilon)=-1$ and if $x$ is a lower connected saddle which is four-arm in $M$, then $N_{\text{simple}}(c-\epsilon)-N_{\text{simple}}(c+\epsilon)=0$.
	
	Finally let $x$ be a tangent point, so that $A_{c-\epsilon}$ is homotopy equivalent to $A_{c+\epsilon}$ with a $k$-cell attached for some $k\in\{0,1,2\}$. In each case, $A_{c+\epsilon}$ has at most two simple components disjoint from $\partial M$, so $\lvert N_{\text{simple}}(c-\epsilon)-N_{\text{simple}}(c+\epsilon)\rvert\leq 2$.
	
	Now suppose that there are no critical or tangent points at level $\ell$. Since $N_\text{simple}(a)\to 0$ as $a\to\infty$, we see that $N_\text{simple}(\ell)$ equals the sum of the finite number of jumps $N_{\text{simple}}(c-\epsilon)-N_{\text{simple}}(c+\epsilon)$ at each level $c$ with a critical/tangent point. This sum equals the number of local maxima of $f$ above level $\ell$ minus the corresponding number of lower connected saddles with an error $\zeta$ bounded in absolute value by $N_\text{4-arm}(M)+2N_\text{tang}(M)$. By Lemma~\ref{l:four arm}, $\lvert\zeta\rvert\leq 5N_\text{tang}(M)$.
	
	Finally, we note that $\{f\geq\ell\}=\cap_{\epsilon>0}\{f\geq\ell-\epsilon\}$ and that the intersection of a decreasing family of compact, connected sets is connected, so that $\{f\geq\ell\}$ and $\{f\geq\ell-\epsilon\}$ have the same number of components for $\epsilon$ sufficiently small. Repeating the earlier argument for tangent points shows that these sets have the same number of components which do not intersect $\partial M$. Then, if $f$ has a critical point $x$ at level $\ell$, applying the Morse lemma on a neighbourhood of $x$ and Theorem \ref{Morse 1} outside this neighbourhood shows that $N_{\text{simple}}(\ell-\epsilon)=N_{\text{simple}}(\ell)$. (If no such point exists, we simply apply Theorem \ref{Morse 1} to $M$.) Therefore we can apply the above arguments to the level $\ell-\epsilon$ such that there are no critical or tangent points with level in $[\ell-\epsilon,\ell)$ to prove the lemma.
\end{proof}

With the above preliminary lemma, the proof of Lemma~\ref{l:main topological} in the case of aperiodic functions is straightforward.

\begin{proof}[Proof (Lemma~\ref{l:main topological} in the case of aperiodic functions)]
	Let $f$ be an aperiodic function satisfying the assumptions of Lemma~\ref{l:main topological}. Applying Lemma~\ref{l:excursion sets manifolds} with $M=\overline{B(R)}$ and $M_\infty=\R^2$ gives exactly the stated relationship for excursion sets. To complete the proof, we prove a corresponding relationship for level sets. We fix a level $\ell$ and let $f_R=f|_{\overline{B(R)}}$, we construct a graph on the vertex set
	\begin{displaymath}
		V:=\{\text{Components of }\{f_R\geq\ell\}\}\cup\{\text{Components of }\{f_R\leq\ell\}\}
	\end{displaymath}
	by declaring two vertices to be joined by an edge if they have non-empty intersection. Clearly the graph is bipartite and each edge corresponds to a component of $\{f_R=\ell\}$. This graph is acyclic, and so by Euler's formula
	\begin{align*}
		\#\{\text{Components of }\{f_R=\ell\}\}=&\#\{\text{Components of }\{f_R\geq\ell\}\}\\
		&+\#\{\text{Components of }\{f_R\leq\ell\}\}-1
	\end{align*}
	The number of components of $\{f_R=\ell\}$ which intersect $\partial B(R)$ is bounded above by the number of tangent points of $f$ in $\partial B(R)$, and the same bound holds for the components of $\{f_R\geq\ell\}$ and $\{f_R\leq\ell\}$. Therefore we can express the equation above as
	\begin{displaymath}
		N_{LS,R}(f,\ell)=N_{ES,R}(f,\ell)+N_{ES,R}(-f,-\ell)+\eta^{(1)}(\ell)
	\end{displaymath}
	where $\lvert\eta^{(1)}(\ell)\rvert\leq 4 N_\text{tang}(\partial B(R))$. Applying the first part of this lemma to each of the $N_{ES,R}$ terms here then completes the proof.
\end{proof}

For the periodic cases, the argument is a little more technical since we cannot apply Lemma~\ref{l:excursion sets manifolds} to $B(R)$ directly. Instead, we tile $B(R)$ with translated parallelograms or rectangles, apply Lemma~\ref{l:excursion sets manifolds} to $f_\mathbb{T}$ or $f_{\mathcal{C}(n)}$ on each translated domain and aggregate the results.

\begin{proof}[Proof (Lemma~\ref{l:main topological} in the case of periodic functions)]
	Let $f$ be a doubly or singly periodic function satisfying the assumptions of Lemma~\ref{l:main topological}. It suffices to prove the excursion set relationship, since the level set relationship will follow by the same argument as in the aperiodic case.
	
	Suppose $f$ is doubly periodic with periodic vectors $y,z\in\R^2$ and associated parallelogram $P$. By assumption, $f_\mathbb{T}$ is a Morse function almost surely. Suppose that $\{f_\mathbb{T}\geq\ell\}$ has $N_\text{simple}$ simple components.
	
	Let $A$ be a component of $\{f\geq\ell\}$ and $A^\prime$ be the corresponding component of $\{f_\mathbb{T}\geq\ell\}$. If $A$ is compact, then clearly $A^\prime$ cannot contain a non-$\mathbb{T}$-contractible loop, and so it is simple. Since $f|_{\partial P}$ has a finite number of local extrema, $\{f\geq\ell\}\cap\partial P$ has a finite number of components. Therefore if $A$ intersects the boundary of a translated parallelogram $P+n_1y+n_2z$ where $n_1,n_2\in\mathbb{Z}$ then it must contain the translation of one of these boundary components. So if $A$ is unbounded, it must contain a path joining two translated versions of the same boundary component, which implies that $A^\prime$ is not simple.
	
	We can tile $\R^2$ with the translated parallelograms $P+n_1y+n_2z$ where $n_1,n_2\in\mathbb{Z}$. We associate each bounded component $A$ of $\{f\geq\ell\}$ to a particular translation $P+n_1y+n_2z$ of $P$ such that $(P+n_1y+n_2z)\cap A\neq\emptyset$, $N_\text{simple}$ components are mapped to each translated parallelogram and if $A$ is associated with $P$, then $A+n_1y+n_2z$ is associated with $P+n_1y+n_2z$. If $A$ is a component of $\{f\geq\ell\}$ contained in $B(R)$ then it must be associated with a parallelogram which intersects $B(R)$, therefore
	\begin{equation}\label{e:tiling circle 1}
		N_{ES,R}(\ell)\leq N_\text{simple}\cdot\#\{\text{Translations of $P$ in }B(R+d)\}
	\end{equation}
	where $d=\text{diam}(P)$. Similarly, each compact component of $\{f\geq\ell\}$ which is associated with a parallelogram inside $B(R)$ must be in $B(R)$ unless it intersects $\partial B(R)$, but the number of such intersections is bounded above by the number of tangent points of $f$ to $B(R)$. Therefore
	\begin{equation}\label{e:tiling circle 2}
		N_\text{simple}\cdot\#\{\text{Translations of $P$ in }B(R)\}-N_\text{tang}(\partial B(R))\leq N_{ES,R}(\ell)
	\end{equation}
	The number of translated parallelograms contained in the ball of a given radius can be approximated by a generalisation of Gauss' circle problem. It is shown in \cite{lax1982asymptotic} that
	\begin{displaymath}
		\#\{\text{Translations of $P$ in }B(R)\}=\frac{\pi R^2}{\text{Area}(P)}+o(R)
	\end{displaymath}
	as $R\to\infty$. We also note that $N_\text{simple}\leq N_\text{crit}(P)$ since each component must contain
	a critical point. Combining these two facts with \eqref{e:tiling circle 1} and \eqref{e:tiling circle 2}
	\begin{displaymath}
		N_{ES}(R,\ell)=N_\text{simple}\cdot\frac{\pi R^2}{\text{Area}(P)}+O(N_\text{crit}(P)\cdot R+N_\text{tang}(\partial B(R)))
	\end{displaymath}
	as $R\to\infty$.
	
	Now suppose that $\{f_\mathbb{T}\geq \ell\}$ contains $m$ local maxima. Reasoning as above, it is clear that
	\begin{align*}
		m\cdot\#\{\text{Translations of $P$ in }B(R)\}&\leq N_{m^+,R}[\ell,\infty)\\
		&\leq m\cdot\#\{\text{Translations of $P$ in }B(R+d)\}
	\end{align*}
	and so
	\begin{displaymath}
		N_{m^+,R}[\ell,\infty)=m\cdot\frac{\pi R^2}{\text{Area}(p)}+O(N_\text{crit}(P)\cdot R)
	\end{displaymath}
	with an analogous result holding for lower connected saddles. Applying Lemma~\ref{l:excursion sets manifolds} to $f_\mathbb{T}$ (i.e.\ with the setting $M_\infty = M = \mathbb{T}$) then shows that
	\begin{displaymath}
		N_{ES,R}(\ell)=N_{m^+,R}[\ell,\infty)-N_{s^-,R}[\ell,\infty))+O(N_\text{crit}(P)\cdot R+N_\text{tang}(\partial B(R)))
	\end{displaymath}
	as $R\to\infty$ where the constant in the $O(\cdot)$ notation is independent of $\ell$, as required (although it may depend on $d$).
	
	Now suppose that $f$ is singly periodic with periodic vector $y=(y_1,0)$ and define $\mathcal{R}(n)=[ny_1,(n+1)y_1)\times\R$. By assumption $f$ almost surely has a finite number of tangent points in $[0,y_1]\times\{n_1\}$ or $\{n_2y_1\}\times[-R,R]$ for any $n_1,n_2\in\mathbb{Z}$ or $R>0$. Repeating the argument from the doubly periodic case then shows that the compact components of $\{f\geq\ell\}$ correspond precisely to the simple components of $\{f_\mathcal{C}\geq\ell\}$.
	
	As before, we can associate each compact component of $\{f\geq\ell\}$ with a translated rectangle $\mathcal{R}(n)$ for $n\in\mathbb{Z}$, in such a way that if $A$ is a compact component of $\{f\geq\ell\}$ which is mapped to $\mathcal{R}(n)$ then $A$ intersects $\mathcal{R}(n)$, and if $A^\prime=A+y$ then $A^\prime$ maps to $\mathcal{R}(n+1)$.
	
	We recall our earlier definitions: if $S(n_1,n_2)=[n_1y_1,(n_1+1)y_1]\times[n_2,n_2+1]$ where $n_1,n_2\in\mathbb{Z}$, then we define $B_\text{int}(x,R)$ to be the union over $n_1,n_2\in\mathbb{Z}$ of all $S(n_1,n_2)$ contained in $B(x,R)$. Similarly we define $B_\text{ext}(x,R)$ to be the union over $n_1,n_2\in\mathbb{Z}$ of $S(n_1,n_2)$ which intersect $B(x,R)$. We now fix $R>0$ and define $m_n$ for $n\in\mathbb{Z}$ to be the largest integer such that
	\begin{displaymath}
		[ny_1,(n+1)y_1]\times[-m_n,m_n]\subset B(R)
	\end{displaymath}
	and note that under the standard quotient relationship, this rectangle can be identified with $\mathcal{C}(m_n)$.
	
	Let $N_\text{simple}(f_{\mathcal{C}(m_n)},\ell)$ denote the number of simple components of $\{f|_{\mathcal{C}(m_n)}\geq\ell\}$ which do not intersect $\partial\mathcal{C}(m_n)$. By the mapping described above, there will be at least $N_\text{simple}(f_{\mathcal{C}(m_n)},\ell)$ compact components of $\{f\geq\ell\}$ which are associated with $\mathcal{R}(n)$ and intersect $[ny_1,(n+1)y_1]\times[-m_n,m_n]$. If any of these components intersect $\partial B_\text{int}(0,R)$, then they will do so at distinct components of $\{f|_{\partial B_\text{int}(0,R)}\geq\ell\}$ and so in particular the number of such components intersecting the boundary is at most the number of tangent points on $\partial B_\text{int}(0,R)$. Therefore
	\begin{displaymath}
		N_{ES}(R,\ell)\geq\sum_n N_\text{simple}(f_{\mathcal{C}(m_n)},\ell)-N_\text{tang}(\partial B_\text{int}(0,R))
	\end{displaymath}
	We can then apply Lemma~\ref{l:excursion sets manifolds} to $f_{\mathcal{C}(m_n)}$ (i.e.\ with the setting $M_\infty = \mathcal{C}$ and $M = \mathcal{C}(m_n)$) for each $n$ to get this inequality in terms of local maxima and lower connected saddles points. Let $r=\max\{\sqrt{2},y_1\}$. Then since $B_\text{int}(0,R)$ covers $B(R-r)$ we see that
	\begin{align*}
		N_{ES}(R,\ell)\geq& N_{m^+}(R,\ell)-N_{s^-}(R,\ell)-6N_\text{tang}(\partial B_\text{int}(R))-N_\text{crit}(B(R)\backslash B(R-r)).
	\end{align*}
	This gives the required lower bound for $N_{ES}(R,\ell)$. The upper bound follows from a very similar argument. First we define $M_n$ to be the largest integer such that
	\begin{displaymath}
		[ny_1,(n+1)y_1]\times[M_n-1,M_n]\subset B_\text{ext}(0,R)
	\end{displaymath}
	so that $B(R)$ is covered by the finite union of rectangles $[ny_1,(n+1)y_1]\times[-M_n,M_n]$. Then 
	\begin{displaymath}
		N_{ES}(R,\ell)\leq\sum_n N_\text{simple}(f_{\mathcal{C}(M_n)},\ell)
	\end{displaymath}
	and by applying Lemma~\ref{l:excursion sets manifolds} to each cylinder we see that
	\begin{align*}
		N_{ES}(R,\ell)\leq& N_{m^+}(R,\ell)-N_{s^-}(R,\ell)+5N_\text{tang}(\partial B_\text{ext}(R))+N_\text{crit}(B(R+r)\backslash B(R))
	\end{align*}
	as required.
\end{proof}

We next complete the proof of Lemma~\ref{l:no infinite four arm}. Again we shall separate the argument into the aperiodic case and the periodic cases, so that the reader interested only in the aperiodic case can access the simplest version of the argument. Since in this case we work only with the manifolds $\mathbb{R}^2$ and $B( R)$, we replace the condition that an excursion set component $A$ is simple with the equivalent condition that it is bounded.

\begin{proof}[Proof (Lemma~\ref{l:no infinite four arm} in the case of aperiodic fields)]
	Let $f$ be an aperiodic stationary field satisfying Conditions~\ref{conditions main results}. Let $g:\R^2\to\R$ be a Morse realisation of $f$ and assume that $g|_{\overline{B(n)}}$ is Morse for all $n\in\mathbb{N}$. Let $x$ be a lower connected saddle of $g$ at level $\ell$.
	
	The first step is to show that, with probability one, $x$ is not also an upper connected saddle of $g$. For $\epsilon>0$ sufficiently small, let $A_{\ell-\epsilon}$ be the component of $\{g\geq\ell-\epsilon\}$ containing $x$ and $A_{\ell+\epsilon}=A_{\ell-\epsilon}\cap\{g\geq\ell+\epsilon\}$. From the definition of a lower connected saddle point, we have three cases to consider. In the first two cases we use the fact that $g$ is Morse to (deterministically) rule out the possibility of $x$ being upper connected; we show that the third case occurs with zero probability and can therefore be neglected.
	
	\textbf{1)$A_{\ell-\epsilon}$ is unbounded and $A_{\ell+\epsilon}$ has one bounded and one unbounded component.}\\
	Let $S$ denote the bounded component of $A_{\ell+\epsilon}$. We start by choosing $n$ sufficiently large that $\overline{B(n)}$ contains a neighbourhood of $S$ and a neighbourhood of $x$. We let $A^\prime_{\ell-\epsilon}$ be the component of $\{g|_{\overline{B(n)}}\geq\ell-\epsilon\}$ containing $x$ and $A^\prime_{\ell+\epsilon}$ denote $A^\prime_{\ell-\epsilon}\cap\{g\geq\ell+\epsilon\}$. We can apply Theorem~\ref{Morse 2} to $-g$ to deduce that $\overline{B(n)}\cap\{g\leq\ell+\epsilon\}$ is homotopy equivalent to $\overline{B(n)}\cap\{g\leq\ell-\epsilon\}$ with a $1$-cell, denoted $\gamma$, attached. By reducing $\epsilon$ we can ensure that $A^\prime_{\ell+\epsilon}$ has two components, one of which is bounded, and so both end-points of $\gamma$ must be contained in the same component of $\overline{B(n)}\cap\{g\leq\ell-\epsilon\}$. This will in turn be contained in a component of $\{g\leq\ell-\epsilon\}$ which we denote by $B$. Clearly $B\cup\gamma$ is bounded if and only if $B$ is, therefore $\overline{B(n)}\cap\{g\leq\ell-\epsilon\}$ and $\overline{B(n)}\cap\{g\leq\ell+\epsilon\}$ have the same number of bounded components. This holds for any $n$ sufficiently large (possibly after reducing $\epsilon$) but if $x$ were upper connected we could find $m$ large enough that $\overline{B(m)}\cap\{g\leq\ell-\epsilon\}$ has one more bounded component than $\overline{B(m)}\cap\{g\leq\ell+\epsilon\}$ (this argument is stated more formally in the proof of Lemma~\ref{l:excursion sets manifolds}). Therefore $x$ is not an upper connected saddle point.
	
	\textbf{ 2) $A_{\ell-\epsilon}$ is bounded and $A_{\ell+\epsilon}$ has two bounded components.} \\
	The arguments in the previous case are also valid in this case where $S$ is chosen to be either of the two components of $A_{\ell+\epsilon}$.
	
	\textbf{3)$A_{\ell-\epsilon}$ is unbounded and $A_{\ell+\epsilon}$ is bounded.}\\
	We will show that when $f$ is aperiodic, it almost surely has no saddle points of this type. We fix $R>0$ and let $g_R=g|_{B(R)}$ which we assume is Morse. For $y\in\R^2$ we define $A_{y,-\epsilon}$ to be the component of $\{g\geq g(y)-\epsilon\}$ containing $y$ and $A_{y,\epsilon}=A_{y,-\epsilon}\cap\{g\geq g(y)+\epsilon\}$. Let $x_1,\dots,x_n$ be the saddle points of $g_R$ for which $A_{x_i,-\epsilon}$ is unbounded but $A_{x_i,\epsilon}$ is bounded for all $\epsilon>0$ sufficiently small. We then choose a fixed $\epsilon$ sufficiently small that this condition holds for each $i$ and such that $g_R$ has no other critical points and no tangent points in $\{g_R(x_i)-3\epsilon\leq g_R\leq g_R(x_i)+3\epsilon\}$ for any $i$. This ensures the intervals $[g(x_i)-\epsilon,g(x_i)+\epsilon]$ are non-overlapping for $i=1,\dots,n$. Finally we assume that the $x_i$ are ordered so that $g(x_i)<g(x_j)$ for $i<j$.
	
	Since $A_{x_i,-\epsilon}$ is unbounded for each $i$, we see that
	\begin{displaymath}
		B_i^-:=A_{x_i,-\epsilon}\cap\partial B(R)
	\end{displaymath}
	is non-empty for each $i$. Since there are no tangent points with level in $[g_R(x_i)-\epsilon,g_R(x_i)+\epsilon]$, we can repeat the arguments in the proof of Lemma~\ref{l:excursion sets manifolds} to show that
	\begin{displaymath}
		B_i^+:=A_{x_i,\epsilon}\cap\partial B(R)
	\end{displaymath}
	is non-empty and has the same number of components as $B_i^-$ for each $i=1,\dots,n$ (see Figure~\ref{Fig_18}). Suppose there exists a point $y\in B_i^-\cap B_j^-$ for some $i<j$. Since $A_{x_j,-\epsilon}$ is unbounded and path connected, there exists an unbounded path started at $y$ and contained in $\{g\geq g(x_j)-\epsilon)\}$. Since $y\in A_{x_i,\epsilon}$ which is path connected, this path must also be contained in $A_{x_i,\epsilon}$ which is bounded. This contradiction implies that $B_1^-,\dots,B_n^-$ are disjoint. Since each of these sets must have a local maximum, we see that $n\leq N_\text{tang}(\partial B(R))$. Since $f$ is stationary, we know that the expected number of lower connected saddle points $x$ contained in $B(R)$ for which $A_{x,-\epsilon}$ is unbounded and $A_{x,\epsilon}$ is bounded for $\epsilon>0$ sufficiently small is $c_0R^2$ for some $c_0\geq 0$. However by the argument just given, this number is bounded above by the expected number of tangent points of $f$ to $\partial B(R)$ which is $O(R)$ by Conditions~\ref{conditions main results}. Hence $c_0=0$ so $f$ almost surely has no saddle points of this type.
	
	\begin{figure*}[h!]
		\centering
		\input{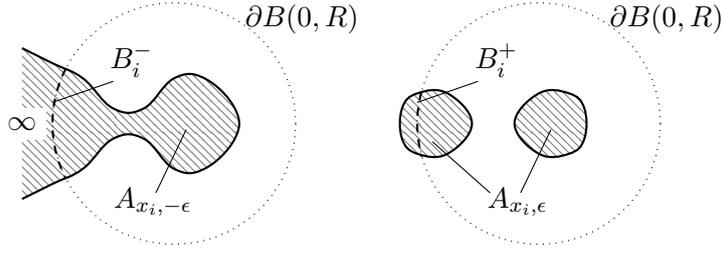}
		\caption{We show that almost surely an unbounded excursion set component of $f$ cannot shrink into a bounded excursion set by passing through a saddle point.}\label{Fig_18}
	\end{figure*}
	
	We have shown that the set of lower connected saddles of $f$ is almost surely disjoint from the set of upper connected saddles. Using the same notation as above, we will now show that any saddle point of $g$ must be lower connected or upper connected, completing the proof of the lemma.
	
	For a small neighbourhood $N$ of $x$, $\{g>\ell\}\cap N$ and $\{g<\ell\}\cap N$ each have two components. Since $g$ is Morse, it has no other critical points at level $\ell$ and so the level set $\{g=\ell\}$ consists of non-intersecting curves except for the component which contains $x$. Since $g$ has no infinite-four-arm saddles there are two mutually exclusive possibilities to consider; either there exists a path in $\{g<\ell\}$ joining the two components of $\{g<\ell\}\cap N$ or there exists a path in $\{g>\ell\}$ joining the two components of $\{g>\ell\}\cap N$.
	
	We now show that in the first case $x$ will be lower connected. By symmetry, this will imply that in the second case, $x$ is upper connected and so complete the proof for aperiodic fields. We fix a path $\gamma$ in $\{g<\ell\}$ connecting the components of $\{g<\ell\}\cap N$. From the existence of $\gamma$ it is clear that $A_\ell\backslash\{x\}$ has two components $B_1,B_2$ and that without loss of generality, $B_1$ is bounded (see Figure~\ref{Fig_19}).
	
	\begin{figure*}[h!]
		\centering
		\input{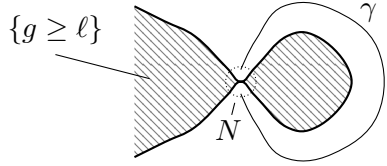}
		\caption{The existence of the path $\gamma$ implies that $x$ is a lower connected saddle point.}
		\label{Fig_19}
	\end{figure*}
	
	For $\epsilon>0$ sufficiently small, $A_{\ell+\epsilon}$ therefore has at least one bounded component $B_1\cap\{g\geq\ell+\epsilon\}$. If $A_{\ell-\epsilon}$ is compact for some $\epsilon>0$ then clearly $B_2$ is also bounded and so $A_{\ell+\epsilon}$ has two bounded components so $x$ is lower connected. If $A_{\ell-\epsilon}$ is unbounded for arbitrarily small $\epsilon$ we can assume that $A_{\ell+\epsilon}$ has an unbounded component, by the argument given above, and so $x$ is again lower connected. We note that this also proves the equivalence of Definition~\ref{d:lower connected general} in the aperiodic case and Definition~\ref{d:lower connected aperiodic}.
\end{proof}

\begin{proof}[Proof (Lemma~\ref{l:no infinite four arm} in the case of periodic fields)] The arguments here are similar to those in the aperiodic case (i.e.\ replacing the condition of boundedness with the more general condition of simplicity). Let $f$ be a stationary field satisfying Conditions~\ref{conditions main results}.
	Let $g:M\to\R$ be a Morse realisation of $f$ restricted to $M$, where $M$ is one of $\mathbb{T}$ or $\mathcal{C}$ depending on whether $f$ is doubly periodic or singly periodic. We also assume that $g|_{\mathcal{C}(n)}$ is Morse for all $n\in\mathbb{N}$ in the singly periodic case. Let $x$ be a lower connected saddle of $g$ at level $\ell$.
	
	Again the first step is to show that, with probability one, $x$ is not an upper connected saddle of $g$. For $\epsilon>0$ sufficiently small, let $A_{\ell-\epsilon}$ be the component of $\{g\geq\ell-\epsilon\}$ containing $x$ and $A_{\ell+\epsilon}=A_{\ell-\epsilon}\cap\{g\geq\ell+\epsilon\}$. Again we have three cases to consider, the first two of which are proven deterministically:
	
	\textbf{1)$A_{\ell-\epsilon}$ is not simple and $A_{\ell+\epsilon}$ has one simple and one non-simple component.}\\
	Let $S$ denote the simple component of $A_{\ell+\epsilon}$. If $f$ is singly periodic, we set $M_c=\mathcal{C}(n)$ where $n$ is sufficiently large that $M_c$ contains a neighbourhood of $S$ and a neighbourhood of $x$. If $f$ is doubly periodic we take $M_c=M=\mathbb{T}$. We define $A^\prime_{\ell-\epsilon}$ to be the component of $\{g|_{M_c}\geq\ell-\epsilon\}$ containing $x$ and $A^\prime_{\ell+\epsilon}:=A^\prime_{\ell-\epsilon}\cap\{g|_{M_c}\geq\ell+\epsilon\}$. By Theorem~\ref{Morse 2} applied to $-g$ we deduce that $M_c\cap\{g\leq\ell+\epsilon\}$ is homotopy equivalent to $M_c\cap\{g\leq\ell-\epsilon\}$ with a $1$-cell, denoted $\gamma$, attached. By reducing $\epsilon$ we can ensure that $A^\prime_{\ell+\epsilon}$ has two components, one of which is simple, and so both end-points of $\gamma$ must be contained in the same component of $M_c\cap\{g\leq\ell-\epsilon\}$. This will then be contained in a component of $\{g\leq\ell-\epsilon\}$ which we denote by $B$. If $B$ is not simple, then clearly $B\cup\gamma$ is not simple. Suppose that $B$ is simple but $B\cup\gamma$ is not simple, then $B\cup\gamma$ must contain a loop, denoted $\eta$, which is not $M$-contractible and in particular $\eta$ must intersect $\gamma$. However $B$ is path connected and $B\cup\gamma$ `surrounds' a region which is homotopy equivalent to $S$ and so must be simple (See Figure~\ref{Fig_20} for an example of these sets in the doubly periodic case).
	
	\begin{figure*}[h!]
		\centering
		\input{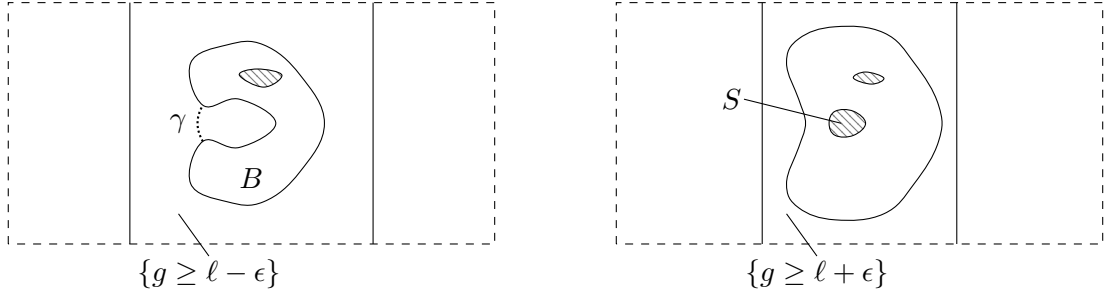}
		\caption{An example of the excursion sets at a level below (left) and above (right) a saddle point at level $\ell$ for which $A_{\ell-\epsilon}$ is not simple and $A_{\ell+\epsilon}$ has one simple and one non-simple component when $M=\mathbb{T}$.}
		\label{Fig_20}
	\end{figure*}
	
	Therefore $\eta$ is homotopy equivalent to a loop which is contained in $B$ and so $B$ is not simple, which is a contradiction. We have shown that $B\cup\gamma$ is simple if and only if $B$ is, therefore $M_c\cap\{g\leq\ell-\epsilon\}$ and $M_c\cap\{g\leq\ell+\epsilon\}$ have the same number of simple components. This holds for any $M_c$ sufficiently large (possibly after reducing $\epsilon$) but if $x$ were upper connected we could find $M_c$ such that $M_c\cap\{g\leq\ell-\epsilon\}$ has one more simple component than $M_c\cap\{g\leq\ell+\epsilon\}$ (again, we note that this argument is stated more formally in the proof of Lemma~\ref{l:excursion sets manifolds}). We conclude that $x$ is not an upper connected saddle point.
	
	\textbf{ 2) $A_{\ell-\epsilon}$ is simple and $A_{\ell+\epsilon}$ has two simple components.} \\
	Once again, the arguments in the previous case are also valid in this case where $S$ is chosen to be either of the two components of $A_{\ell+\epsilon}$.
	
	\textbf{3)$A_{\ell-\epsilon}$ is not simple and $A_{\ell+\epsilon}$ is simple (in particular, connected).}\\
	We can repeat the argument given in this section of the proof for aperiodic fields to show that if $f$ is singly periodic then $A_{\ell-\epsilon}$ must be bounded. Clearly in the doubly periodic case, $A_{\ell-\epsilon}\subset\mathbb{T}$ must also be bounded. We therefore choose a compact domain $M_c\subset M$ of the form $\mathbb{T}$ or $\mathcal{C}(n)$ for some $n$ which contains a neighbourhood of $A_{\ell-\epsilon}$. Let $B_1,\dots,B_n$ denote the components of $\{g\leq\ell+\epsilon\}\cap M_c$ which intersect $A_{\ell+\epsilon}$. Since $A_{\ell+\epsilon}$ is simple, we know that at most one of these sets is not simple and the remainder are. (One of the sets will `surround' $A_{\ell+\epsilon}$ which in turn will `surround' the remaining sets.) Without loss of generality, we assume $B_1$ is the `surrounding' set, which may not be simple. By Theorem~\ref{Morse 2}, $A_{\ell-\epsilon}$ is homotopy equivalent to $A_{\ell+\epsilon}$ with a $1$-cell, denoted $\gamma$ attached. Clearly $\gamma$ is contained in $B_1$ and $B_1$ is not simple, otherwise $A_{\ell+\epsilon}\cup\gamma$ would be simple. If $M=\mathcal{C}$ then $B_1\backslash\gamma$ has two components and each of these components contains a loop which is not $\mathcal{C}$-contractible (since $A_{\ell+\epsilon}\cup\gamma$ is compact and so separated from $\infty$ and $-\infty$ by $B_1\backslash\gamma$). If $M=\mathbb{T}$ then $B_1\backslash\gamma$ may have one or two components, in either case each component will contain a non-$\mathbb{T}$-contractible loop (to see this, consider a path on either side of $\gamma$ which then traverses the boundary of $A_{\ell+\epsilon}$). This means that passing through the saddle point $x$ can only create non-simple components of $\{g\leq\ell-\epsilon\}$, so $x$ is not upper connected.
	
	This completes the proof that the sets of upper and lower connected saddle points of $f$ are almost surely disjoint. We continue to use the notation defined above, and we now show that any saddle point of $g$ must be lower connected or upper connected. As in the aperiodic case, Lemma~\ref{l:four arm} and Conditions~\ref{conditions main results} allow us to conclude that $f$ almost surely has no infinite four-arm saddles.
	
	Suppose that $f$ is doubly periodic so that $g$ is defined on the torus $\mathbb{T}$. We fix a small neighbourhood $N$ of $x$ such that $\{g>\ell\}\cap N$ has precisely two components denoted $N_1,N_2$. If $\gamma:[0,1]\to\mathbb{T}$ is a path contained in $\{g>\ell\}\cup\{x\}$, we say that $\gamma$ \textit{cuts through} $x$ if $x\in\gamma[0,1]$ and for every $t\in(0,1)$ such that $\gamma(t)=x$, for $\epsilon>0$ sufficiently small $\gamma((t-\epsilon,t))\subset N_1$ and $\gamma((t,t+\epsilon))\subset N_2$ (intuitively, this means that the image of $\gamma$ just before hitting $x$ is always on the same side of the saddle point; see Figure~\ref{Fig_21}). We make an analogous definition for paths contained in $\{g<\ell\}\cup\{x\}$.
	
	\begin{figure*}[h!]
		\centering
		\input{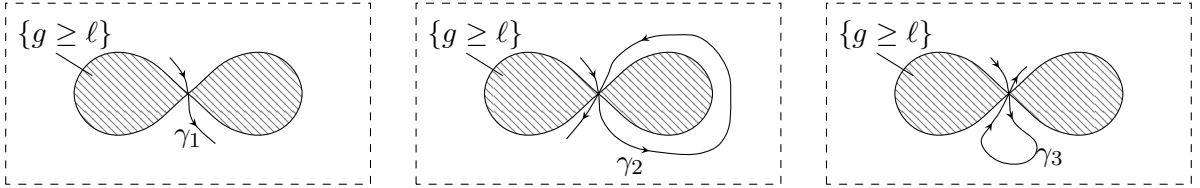}
		\caption{The paths $\gamma_1$ and $\gamma_2$ cut through $x$ but $\gamma_3$ does not.}
		\label{Fig_21}
	\end{figure*}
	
	First we suppose that there exists a $\mathbb{T}$-contractible loop $\gamma$ in $\{g<\ell\}\cup\{x\}$ which cuts through $x$. Since $x$ is a saddle point, it is clear that $A_{\ell+\epsilon}$ must have two components, one of which is surrounded by $\gamma$ and so is simple. If $A_{\ell-\epsilon}$ is simple, then both components of $A_{\ell+\epsilon}$ must be simple. If $A_{\ell-\epsilon}$ is not simple, we know from Lemma~\ref{l:excursion sets manifolds} that $A_{\ell+\epsilon}$ can contain at most one simple component. In either case, we see that $x$ is lower connected. By symmetry, $x$ is upper connected if there exists a $\mathbb{T}$-contractible loop in $\{g>\ell\}\cup\{x\}$ cutting through $x$.
	
	Now we suppose that all loops cutting through $x$ are non-$\mathbb{T}$-contractible. We define $C_+$ to be the union of $\{x\}$ and the components of $\{g>\ell\}$ which $x$ is in the closure of. We define $C_-$ to be the analogous set for $\{g<\ell\}$. At least one of $C_+$ or $C_-$ must contain a non-contractible loop cutting through $x$ (in order to stop any paths in $C_+$ cutting through $x$ from joining to form a loop, there must be a loop in $\mathbb{T}\backslash C_+$ cutting through $x$ which blocks them, and since $x$ is the only critical point at level $\ell$, such a `blocking' loop can be found in $C_-$). We also note that both $C_+$ and $C_-$ must contain non-$\mathbb{T}$-contractible loops (if $C_+$ was simple, then we could find a $\mathbb{T}$-contractible loop in $C_-$ cutting through $x$, and this argument is symmetric).
	
	Suppose that $\gamma_1$ is a loop contained in $C_+$ which cuts through $x$. If every non-$\mathbb{T}$-contractible loop in $C_+$ intersects $x$, then $C_+\backslash\{x\}$ is simple, so $A_{\ell+\epsilon}$ is simple and hence $x$ is lower connected. Therefore we may assume that $C_+$ contains a non-$\mathbb{T}$-contractible loop $\gamma_2$ which does not intersect $x$. Let $\eta_1$ and $\eta_2$ be two loops which generate the fundamental group of $\mathbb{T}$, we denote the concatenation of $n$ copies of $\eta_1$ and $m$ copies of $\eta_2$ by $n\eta_1+m\eta_2$. We suppose $\gamma_1\simeq n_1\eta_1+m_1\eta_2$ and $\gamma_2\simeq n_2\eta_1+m_2\eta_2$ where $\simeq$ denotes being $\mathbb{T}$-homotopic. It is known (see \cite[Section 1.2.3]{farb2011primer}) that any loop homotopic to $n_1\eta_1+m_1\eta_2$ must intersect any loop homotopic to $n_2\eta_1+m_2\eta_2$ unless $n_1m_2=n_2m_1$.
	
	If $n_1m_2\neq n_2m_1$ then we know that any non-contractible loop in $C_-$ must intersect either $\gamma_1$ or $\gamma_2$. Since $C_-\cap C_+=\{x\}$ clearly such a loop must intersect $\gamma_1$ at $x$. By the previous paragraph this means that $C_-\backslash\{x\}$ is simple and so $x$ is upper connected.
	
	If $n_1m_2=n_2m_1$ then by choosing a path $\xi$ in $C_+\backslash\{x\}$ which joins $\gamma_1$ to $\gamma_2$ we consider the concatenated loop
	\begin{displaymath}
		n_2\gamma_1+\xi+ n_1(-\gamma_2)+(-\xi)
	\end{displaymath}
	where $-$ denotes inverting the direction of the path. This path is contained in $C_+$ and cuts through $x$ (since $\gamma_1$ does but $\gamma_2$ and $\xi$ do not hit $x$) and since $n_1m_2=n_2m_1$ this loop is $\mathbb{T}$-contractible. However this contradicts the above supposition, so this case is not possible. This completes the proof in the doubly periodic case. The proof in the singly periodic case is simply a repetition of parts of the proof for aperiodic and doubly periodic fields, so we omit it.
\end{proof}

All that remains is to complete the calculations for the special class of degenerate fields. 

\begin{proof}[Proposition~\ref{p:cilleruello c_NS}]
	We recall that a random variable $Y$ is Rayleigh distributed with parameter $\sigma>0$ if
	\begin{displaymath}
		\mathbb{P}(Y\leq x)=1-e^{-x^2/(2\sigma^2)}
	\end{displaymath}
	for all $x\geq 0$, and we denote this as $Y\sim\text{Ray}(\sigma)$. If $Y\sim\text{Ray}(\sigma)$ and $\theta$ is an independent uniform-$[0,2\pi]$ random variable then it is well known that $Y\cos(\theta)\sim\mathcal{N}(0,\sigma^2)$.
	
	Let $f$ be the Gaussian field with spectral measure
	\begin{displaymath}
		\rho=\alpha\delta_0+\frac{\beta}{2}(\delta_K+\delta_{-K})+\frac{\gamma}{2}(\delta_L+\delta_{-L})
	\end{displaymath}
	where $\alpha+\beta+\gamma=1$ so that the covariance function of $f$ is
	\begin{displaymath}
		\kappa(x)=\alpha+\beta\cos(2\pi K\cdot x)+\gamma\cos(2\pi L\cdot x).
	\end{displaymath}
	Then $f$ has the representation
	\begin{equation}\label{e:rayleigh representation}
		f(x)=X_0+Y_1\cos(2\pi K\cdot x+\theta_1)+Y_2\cos(2\pi L\cdot x+\theta_2)
	\end{equation}
	where $X_0\sim\mathcal{N}(0,\alpha)$, $Y_1\sim\text{Ray}(\sqrt{\beta})$, $Y_2\sim\text{Ray}(\sqrt{\gamma})$, the $\theta_i$ are uniformly distributed on $[0,2\pi]$ and all of these random variables are independent. (By well known properties of the Rayleigh distribution, the field defined by the right hand side of \eqref{e:rayleigh representation} will have Gaussian finite dimensional distributions and simple calculations show that the covariance function of this field is $\kappa$.)
	
	Let 
	\begin{displaymath}
		u=\frac{1}{K_1L_2-K_2L_1}\begin{pmatrix}
			L_2\\
			-L_1
		\end{pmatrix}
		\qquad
		v=\frac{1}{K_1L_2-K_2L_1}\begin{pmatrix}
			K_2\\
			-K_1
		\end{pmatrix}
	\end{displaymath}
	and $P=\{tu+sv:t,s\in[0,1]\}$. Then $P$ is the parallelogram associated with $f$ and we consider $f_\mathbb{T}$ to be the restriction of $f$ to $P$ when $P$ is identified with the two dimensional torus. By rotating the axes, we may assume that $K_1>0$ (since this does not affect the definition of $c_{LS}$). Some basic calculations show that, on the event
	\begin{displaymath}
		\{Y_1\neq 0\}\cap\{Y_2\neq 0\}\cap \left\{Y_1\pm Y_2\neq 0\right\},
	\end{displaymath}
	$f_\mathbb{T}$ has four critical points which occur at different levels and are all non-degenerate. Therefore $f_\mathbb{T}$ is almost surely a Morse function. Moreover, the critical points of $f_\mathbb{T}-X_0$ occur at levels $-Y_1-Y_2,-\lvert Y_1-Y_2\rvert,\lvert Y_1-Y_2\rvert,Y_1+Y_2$. Clearly the critical points at levels $Y_1+Y_2$ and $-Y_1-Y_2$ are a local maximum and local minimum respectively. We can use Lemma~\ref{l:excursion sets manifolds} to characterise the other two critical points and the number of simple components of $\{f_\mathbb{T}-X_0\geq\ell\}$, denoted $N_\text{simple}(\ell)$, at different levels. Specifically, since $\{f_\mathbb{T}-X_0> Y_1+Y_2\}=\emptyset$ and $\{f_\mathbb{T}-X_0\geq -Y_1-Y_2\}=\mathbb{T}$, we see that $N_\text{simple}(\ell)=0$ whenever $\lvert\ell\rvert>Y_1+Y_2$. Then, since the critical points at level $Y_1+Y_2$ and $-Y_1-Y_2$ must be a local maximum and local minimum respectively, applying Lemma~\ref{l:excursion sets manifolds} shows that $N_\text{simple}(\ell)$ does not change as $\ell$ passes through $-Y_1-Y_2$ and decreases by one as $\ell$ passes through $Y_1+Y_2$. Therefore
	\begin{displaymath}
		N_\text{simple}(\ell)=\begin{cases}
			0, &\text{if }\ell>Y_1+Y_2,\\
			1, &\text{if }\ell\in(\lvert Y_1-Y_2\rvert,Y_1+Y_2),\\
			0, &\text{if }\ell\in(-Y_1-Y_2,-\lvert Y_1-Y_2\rvert),\\
			0, &\text{if }\ell<-Y_1-Y_2.
		\end{cases}
	\end{displaymath}
	Now assume that $\theta_1=\theta_2=0$. Then by traversing the parallelogram $P$ across one of its edges (depending on which of $Y_1,Y_2$ is bigger) we can find a closed path which is not $\mathbb{T}$-contractible, on which $f_\mathbb{T}-X_0$ is bounded below by $\lvert Y_1-Y_2\rvert$, so in particular is positive. This shows that $\{f_\mathbb{T}-X_0\geq 0\}$ has a non-simple component (for general values of $\theta_1,\theta_2$, such a path will also exist, but it will be translated). Since $\{f_\mathbb{T}-X_0\geq \lvert Y_1-Y_2\rvert\}$ consists of a single simple component, this implies that the critical point at level $\lvert Y_1-Y_2\rvert$ must be a lower connected saddle point. Then by Lemma~\ref{l:excursion sets manifolds} we see that
	\begin{displaymath}
		N_\text{simple}(\ell)=\begin{cases}
			1, &\text{if }\ell\in(\lvert Y_1-Y_2\rvert,Y_1+Y_2],\\
			0, &\text{otherwise.}
		\end{cases}
	\end{displaymath}
	It is then clear that
	\begin{displaymath}
		p_{m^+}(x)=\frac{1}{\text{Area}(P)}p_{X_0+Y_1+Y_2}(x)\quad \text{and} \quad p_{s^-}(x)=\frac{1}{\text{Area}(P)}p_{X_0+\lvert Y_1-Y_2\rvert}(x) .
	\end{displaymath}
	Arguments given in the proof of Lemma~\ref{l:main topological} for periodic functions show that
	\begin{displaymath}
		N_{ES,R}(\ell)=\id_{\ell-X_0\in(\lvert Y_1-Y_2\rvert,Y_1+Y_2]}\cdot\frac{\pi R^2}{\text{Area}(P)}+O(R+N_\text{tang}(\partial B(R))) 
		.\end{displaymath}
	Arguing as in the proof of Lemma~\ref{Gaussian fields main theorem} it can be shown that $\mathbb{E}(N_\text{tang}(\partial B(R))=O(R)$ as $R\to\infty$. Then for $\ell\in\R$
	\begin{align*}
		\mathbb{E}\left(\left\lvert\frac{N_{ES,R}(\ell)}{\pi R^2}-\frac{1}{\text{Area}(P)}\id_{\ell-X_0\in(\lvert Y_1-Y_2\rvert,Y_1+Y_2]}\right\rvert\right)=O\left(\frac{1}{R}\right) ,
	\end{align*}
	so that $N_{ES,R}(\ell)/(\pi R^2)$ converges in $L^1$ to the random variable $\frac{1}{\text{Area}(P)}\id_{\ell-X_0\in(\lvert Y_1-Y_2\rvert,Y_1+Y_2]}$ which is non-deterministic provided $\alpha>0$ or $\ell>0$. In particular, this convergence shows that
	\begin{align*}
		c_{ES}(\ell):=&\lim_{R\to\infty}\frac{1}{\pi R^2}\mathbb{E}(N_{ES,R}(\ell))
		=\frac{1}{\text{Area}(P)}\mathbb{P}(\ell-X_0\in(\lvert Y_1-Y_2\rvert,Y_1+Y_2]) .
	\end{align*}
	We can repeat these arguments for the sum of upper and lower excursion sets to derive the corresponding expression for $c_{LS}(\ell)$.
\end{proof}

\bigskip
\printbibliography

\end{document}